\documentclass[10pt]{amsart}

\usepackage{xy, graphicx, color, hyperref,array, mathtools}

\usepackage{amsmath}
\usepackage{amssymb}
\usepackage{amsfonts}
 
\usepackage{mathrsfs}

\usepackage{tikz,verbatim}
 \usepackage{dsfont}

\makeatletter
\def\subsection{\@startsection{subsection}{3}%
  \z@{.9\linespacing\@plus.7\linespacing}{.1\linespacing}%
  {\normalfont\bfseries}}
\makeatother 
 
 \xyoption{all}

\title[]{Central Extensions and Cohomology}
 
 \author{Pranjal Jain, Rohit Joshi, Steven Spallone }

 \newtheorem{thm}{Theorem}[section]
\newtheorem{c.intro}[thm]{Corollary}
\newtheorem{lemma}[thm]{Lemma}
\newtheorem{prop}[thm]{Proposition}
\newtheorem{cor}{Corollary}[thm]

\theoremstyle{definition}
\newtheorem{remark}[thm]{Remark}
 
\newtheorem{defn}[thm]{Definition} 
\newtheorem{example}[thm]{Example} 
 
\newcommand{\nc}{\newcommand}

\nc{\mc}{\mathcal}
\nc{\mb}{\mathbb}
\nc{\mf}{\mathfrak}

\nc{\ms}{\mathscr}

\nc{\ul}{\underline}
\nc{\ol}{\overline}
\nc{\N}{\mb N}

\nc{\R}{\mb R}
\nc{\Z}{\mb Z}
 
\nc{\C}{\mb C}

\nc{\sts}[1] {{\color{cyan}\textbf{SS:} [#1]}}
\nc{\pj}[1]{{\color{magenta}\textbf{PJ:} [#1]}}
\nc{\bks}{\backslash}

\nc{\dmo}{\DeclareMathOperator}

\nc{\mat}[4]{
    \begin{pmatrix}
      #1 & #2 \\
      #3 & #4
    \end{pmatrix}
}

\dmo{\Ker}{Ker} 
\dmo{\Pin}{Pin}
 \dmo{\Aut}{Aut}
 \dmo{\Borel}{Borel}
\dmo{\Sq}{Sq}
\dmo{\Ext}{Ext}
\dmo{\odd}{odd}
\dmo{\sgn}{sgn}
\dmo{\cl}{cl}
\dmo{\Cl}{Cl}
\nc{\beq}{\begin{equation*}}
\nc{\eeq}{\end{equation*}}
\nc{\half}{\frac{1}{2}}
\dmo{\id}{id}
\dmo{\diag}{diag}
\dmo{\gp}{gp}
\dmo{\Mod}{mod}
\dmo{\pr}{pr}
\dmo{\GL}{GL}
\dmo{\res}{res}
\dmo{\lin}{lin}
 \dmo{\vol}{vol}
\dmo{\Sp}{Sp}
\dmo{\SO}{SO}
\dmo{\BSO}{BSO}
\dmo{\im}{im}
\dmo{\BO}{BO}

\dmo{\Or}{O}

\dmo{\SL}{SL}
\dmo{\ab}{ab}
 \dmo{\Spin}{Spin}
 
\nc{\la}{\lambda}
  \nc{\eps}{\varepsilon}
  
 \nc{\lip}{\langle}
 \nc{\rip}{\rangle}
\nc{\gm}{\gamma}

\dmo{\Perm}{Perm}
\dmo{\Res}{Res}
\dmo{\Ind}{Ind}
\dmo{\tr}{tr}
\dmo{\Sym}{Sym}
\dmo{\reg}{reg}
\dmo{\End}{End}
\dmo{\Hom}{Hom}
\dmo{\Int}{Int}
\dmo{\Bun}{{\bf Bun}}
\dmo{\bun}{Bun}
\dmo{\HH}{H}
 \DeclareFontFamily{U}{cbgreek}{}
\DeclareFontShape{U}{cbgreek}{m}{n}{
        <-6>    grmn0500
        <6-7>   grmn0600
        <7-8>   grmn0700
        <8-9>   grmn0800
        <9-10>  grmn0900
        <10-12> grmn1000
        <12-17> grmn1200
        <17->   grmn1728
      }{}

\DeclareRobustCommand{\qoppa}{%
  \text{\usefont{U}{cbgreek}{\normalorbold}{n}\symbol{21}}%
}
 
\makeatletter
\newcommand{\normalorbold}{%
  \ifnum\pdf@strcmp{\math@version}{bold}=\z@ bx\else m\fi
}

\address{University of British Columbia, Vancouver, BC, Canada V6T1Z4}
\email{pranjal.jain@math.ubc.ca}
\address{Indian Institute of Science Education and Research, Pune-411021, Maharashtra, India}
\email{rohitsj@students.iiserpune.ac.in}
\address{Indian Institute of Science Education and Research, Pune-411021, Maharashtra, India}
\email{sspallone@gmail.com}

\begin{document}
\maketitle
\begin{center} \today
\end{center}
 
 \tableofcontents
 
 \begin{abstract} 
 Let $G$ be a group which is topologically a CW-complex, $BG$ a classifying space for $G$, and $A$ a discrete abelian group. To a central extension of $G$ by $A$, we associate a cohomology class in $\HH^2(BG,A)$. We prove this association is injective, and bijective in many cases.  A homomorphism of such groups lifts to a central extension iff the pullback of the associated cohomology class vanishes.
 \end{abstract}

 \section{Introduction}
 
 
Let $A$ be a discrete abelian group. When $G$ is a discrete group, there is a well-known natural correspondence between central extensions of $G$ by $A$ and the quadratic group cohomology $\HH^2(G,A)$, attributed to Schreier \cite{Schreier}.
From this, one can systematically find cohomological criteria for lifting problems. In this paper we extend this correspondence to the topological group setting, when $G$ is topologically a CW-complex, for instance   a Lie group.
 
Let $\mb E(G,A)$ be the set of (equivalence classes of) central extensions of $G$ by $A$. A construction of Milgram-Steenrod gives a natural map 
\beq
\alpha_G: \mb E(G,A) \to \HH^2(BG,A),
\eeq
where $BG$ is a classifying space for $G$. When $G$ is discrete, the map $\alpha_G$ agrees with the classical bijection. In this note we prove:

\begin{thm} The map $\alpha_G$ is an injection. It is an isomorphism when $G$ is connected. When $G$ is the semidirect product of a discrete group and a connected group, then $\alpha_G$ is an isomorphism.
\end{thm} 
From the injectivity of $\alpha_G$ we deduce the following lifting criterion.

\begin{thm} \label{mid.intro}
 Let $\varphi:G' \to G$ be a homomorphism, $A$ a discrete abelian group, and  $p: E \to G$ an extension of $G$ by $A$.  Then $\varphi$ lifts to $E$ iff 
\beq
\varphi^*(\alpha_G(E,i,p))=0.
\eeq
(Here $i: A \to E$ is the inclusion.)
\end{thm}

For the orthogonal group $G=\Or(V)$ of a quadratic space, we describe its three inequivalent double covers, denoted by $\Pin^{\pm}(V)$ and $\widetilde \Or(V)$, and then describe $\alpha_G$ for each. 
 Finally, we give an application of the above to Stiefel-Whitney classes of orthogonal representations. 
 
\begin{thm} \label{last.intro} Let $G$ be a topological group which is a CW-complex. An orthogonal representation $\pi: G \to \Or(V)$ lifts to
\begin{enumerate}
\item $\widetilde \Or(V)$ iff $w_1(\pi)^2=0$.
\item $\Pin^+(V)$ iff $w_2(\pi)=0$.
\item $\Pin^-(V)$ iff $w_2(\pi)+w_1(\pi)^2=0$.
\end{enumerate}
When $\pi$ lifts, the set of lifts of $\pi$ admits a simply transitive action by the group of continuous linear characters $\chi: G \to \{ \pm 1\}$. 
\end{thm}


 The equivalence of $\mb E(G,A)$ and $\HH^2(BG,A)$ was observed in the '70s. A proof is given in \cite[Theorem 4]{Wigner} by comparing two spectral sequences in simplicial sheaf cohomology. We have nonetheless written this manuscript because  for a result of this importance, it is useful to have alternate proofs. When $G$ is connected, our correspondence essentially agrees with the obstruction class construction in \cite{NWW}, according to Remark 7.17 in that paper and Proposition \ref{alpha.beta} below. 
 
 The layout of the paper is as follows.  After reviewing definitions for topological groups in Section \ref{notprem}, we examine the category of their central extensions in Section \ref{exs.section}.  
 The theory of (principal) $G$-bundles is reviewed in Section \ref{G.bund.sec}.
  Then Section \ref{class.space.section} reviews some technical bits of topological group theory. Next, Section \ref{Milgram.Steenrod.sec} lays out the Milgram-Steenrod model and its functorial properties. 
   In Section \ref{long.ex.seq.section} we define a correspondence $\beta_G$, rather than $\alpha_G$, when $G$ is connected by means of covering space theory. 
 In Section \ref{Connecting Homomorphisms} we introduce a universal $A$-bundle $EE \to X_{\ms E}$ adapted to the extension $\ms E$, which serves as a kind of hub of the computations.  For disconnected groups, we incorporate the exact sequence for the component group $\pi_0(G)$ in Section \ref{disc.section}. In Section \ref{kappa.section} we use the $K(A,2)$-spaces to define a natural map from the bundles of Section \ref{class.space.section} to the quadratic cohomology. By Section \ref{alpha.section} everything is in place for our definition of $\alpha_G$, and Theorem \ref{mid.intro} follows by gathering the earlier work. 

 Then we focus on the orthogonal groups. In Section \ref{SWC.section} we review Stiefel-Whitney classes, and in Section \ref{double.covs.ort.sec} we describe the double covers of the orthogonal group, culminating in Theorem \ref{last.intro}.
  
  Finally, in Section \ref{sec:further} we indicate extensions of this approach to higher degrees, and to nondiscrete $A$, as investigated in \cite{Pranjal.Thesis}.
  
  \bigskip

\textbf{Acknowledgements.} The first author was supported by the INSPIRE-SHE scholarship.The second author was supported by a postdoctoral fellowship from NBHM (National Board of Higher Mathematics), India. 
We thank  Karthik Vasisht for helpful discussions.
 \section{Notation and Preliminaries} \label{notprem}
 
  \begin{defn}
 We say that a topological space $B$ is \emph{nice}, provided it is Hausdorff, locally path-connected, semilocally simply connected, and nonempty.
\end{defn}
 
 When $G$ is a group, we write $1_G \in G$ for the identity, and $G^{\ab}$ for the abelianization of $G$. When $A$ is an abelian group, we may use additive notation, so the identity would be $0$.
If $x,g \in G$, put
 \beq
 \Int(x)(g)=x^{-1}gx.
 \eeq
  
 Let $\mu_n$ be the group of $n$th roots of unity in $\C^\times$. Let $C_n$ be the cyclic group of order $n$.
 Throughout this paper $A$ is a discrete abelian group.

  For us, topological groups will always be Hausdorff. For a topological group $G$, denote by $G^\circ$ the path component of $G$ containing the identity. It is a normal subgroup of $G$; write $\pi_0(G)=G^\circ \bks G$ for the quotient. Write $\Aut(G)$ for the (continuous) automorphisms of $G$.

 \section{Central Extensions} \label{exs.section}
 
 In this section we study the category of central extensions of topological groups by discrete groups. Throughout, $G$ is a topological group.
 
 
 
\begin{defn} A \emph{central extension of $G$ by $A$} is triplet $\ms E=(E,i,p)$, where $E$ is a topological group, $i: A \to E$ and  $p: E \to G$ are continuous homomorphisms. We require that $i$ is an injection onto a discrete central subgroup of $E$, that $p$ is open and surjective, and that $\ker p=\im i$.
  \end{defn}
Thus $\ms E$ can be identified with the exact sequence
\begin{equation} \label{first.ext}
\mathscr E: 0 \to A \overset{i}{\to} E \overset{p}{\to} G \to 1,
\end{equation}
of continuous homomorphisms of topological groups. Another central extension
\beq
\ms E': 0 \to A \overset{i'}{\to} E' \overset{p'}{\to} G \to 1
\eeq
is \emph{equivalent} to \eqref{first.ext} when there is a homomorphism $\varphi: E \to E'$ with $\varphi \circ i=i'$ and $p' \circ \varphi=p$. (It is necessarily an isomorphism.)
 In this case we write $\ms E \cong \ms E'$. Let ${\bf E}(G,A)$ be the category of central extensions of $G$ by $A$, and let $\mb E(G,A)$ be the set of central   extensions of $G$ by $A$, modulo equivalence.

 \begin{example} \label{univ.gp} Suppose $G$ is a nice connected group. According to  \cite[Proposition 6, Page 379]{Bou.AT}, there is a morphism $p: \tilde G \to G$, with $\tilde G$ a simply connected topological group. There is an injection $i: \pi_1(G) \to \tilde G$ whose image is the kernel of $p$, so that $\ms E_G=(\tilde G,i,p) \in \mb E(G, \pi_1(G))$.
 \end{example}

 \subsection{Functorial Properties}
 
 A  morphism $\varphi: G' \to G$ gives rise to a ``pullback'' morphism  $ \varphi^*:\mb E(G,A) \to \mb E(G',A)$ as follows:
Given $\ms E=(E,i,p) \in {\bf E}(G,A)$, set
  \beq
 E' = \{(g',e) \mid \varphi(g') = p(e) \},
  \eeq
 viewed as a closed subgroup of the product $G' \times E$.
  Define $i':A \to E'$  by $i'(a)=(1,i(a))$, and let $p'$ be the first projection. Finally put $\varphi^* \ms E=(E',i',p')$.
  In particular, the group $\Aut(G)$  acts on $\mb E(G,A)$ by pullback.

  Similarly, a homomorphism $\psi:A \to A'$  between discrete abelian groups gives rise to a ``pushout central extension'' as follows.  Given $\ms E=(E,i,p) \in {\bf E}(G,A)$, set 
  \beq
  A_\psi= \{(i(a),-\psi(a)) \mid a \in A\},
  \eeq
  a normal subgroup of $E \times A'$, and let
  \begin{equation} \label{e.prime}
 E'= (E \times A')/ A_{\psi}.
 \end{equation}
  Define $i': A' \to E'$ to be the evident inclusion, and $p':E' \to G$ by sending the coset of $(e,a') \mod A_{\psi}$ to $p(e)$. Then
  $\psi_* \ms E=(E',i',p')$ is a central extension of $G$ by $A'$, and this gives a homomorphism $\psi_*: \mb E(G,A) \to \mb E(G,A')$.
 
If we define $\psi_E: E \to E'$ by $\psi_E(e)=(e,1) \mod A_{\psi}$, then the diagram
\begin{equation} \label{diag.ek}  
	\xymatrix{
		A \ar[rr]^{\psi} \ar[d]^i &&A' \ar[d]^{i'}\\
		E \ar[rr]^{\psi_E}   \ar[rd]^p && E' \ar[ld]^{p'}  \\
		 & G & \\
	}
	\end{equation}
commutes.

\begin{lemma} \label{likely.commutes} If $\varphi:G' \to G$ and $\psi: A \to A'$ are group homomorphisms, then 
\beq
\psi_* \circ \varphi^*=\varphi^* \circ \psi_*: \mb E(G, A) \to \mb E(G',A').
\eeq
\end{lemma}

\begin{proof} Elementary; see for instance \cite[Exercise 5, page 114]{Maclane.Homology}.
\end{proof}
  
\begin{lemma} \label{liftingtheta} 
	Let $\theta \in \Aut(G)$ and $\ms E =(E,i,p) \in {\bf E}(G,A)$. Then $\theta^*(\ms E) \cong \ms E$ iff there exists    $\tilde{\theta} \in \Aut(E)$ so that
	\beq
	\xymatrix{
		E \ar[r]^{\tilde{\theta}} \ar[d]^p& E \ar[d]^p\\
		G \ar[r]^{\theta} & G\\
	}
	\eeq 
	commutes.
	
\end{lemma}
 
\begin{proof} Since $\theta^*(\ms E) \cong (E,i,\theta^{-1} \circ p)$, an equivalence between $\ms E$ and $\theta^*(\ms E)$ amounts to the existence of such a $\tilde \theta$.
\end{proof}

\begin{cor} \label{inner.triv.action} Given $x \in G$, the inner automorphism $\Int(x)$ acts trivially on $\mb E(G,A)$.
\end{cor}

\begin{proof} Let $\ms E =(E,i,p) \in {\bf E}(G,A)$. Pick $y \in E$ with $p(y)=x$. It is easy to see that $\Int(y)$ covers $\Int(x)$, so $\Int(x)^*(\ms E) \cong \ms E$ by Lemma \ref{liftingtheta}.
\end{proof}

\subsection{Baer Sums} \label{Baer.groups} 
Let $\ms E_1,\ms E_2 \in {\bf E}(G,A)$, with $\ms E_k=(E_k,i_k,p_k)$ for $k=1,2$. Write $E_3'$ for the pullback of $p_1$ and $p_2$:
\beq
	\xymatrix{
		E_3' \ar[r] \ar[d] &E_2 \ar[d]^{p_2}\\
		E_1 \ar[r]^{p_1}   & G  \\
	}
	\eeq 
Write $\triangledown: A \to E_3'$ for the antidiagonal map $ \triangledown(a)=(i_1(a),i_2(a)^{-1})$, and  let $E_3$ be the quotient of $E_3'$ by the image of $\triangledown$.
Define $i_3:A \to E_3$ by $i_3(a)=(a,1) \mod \triangledown(A)$, and $p_3:E_3 \to G$ by $p_3(e_1,e_2)=p_1(e_1)$. Then 
$\ms E_3=(E_3,i_3,p_3) \in {\bf E}(G,A)$, and is called the \emph{Baer sum} $\ms E_3=\ms E_1 + \ms E_2$
of $\ms E_1$ and $\ms E_2$. The Baer sum makes $\mb E(G,A)$ into an abelian group.  (See \cite[Definition 3.4.4]{weibel}.)  The neutral element is represented by the evident trivial extension
\beq
\mathscr E_0:  0 \to A \to G \times A \to G \to 1.
\eeq

 \subsection{Relation to Lifting Problems} \label{rel.lift.prob}
 The following is clear:
 \begin{prop} \label{ext.lift} Let $\varphi:G' \to G$ and $\ms E=(E,i,p) \in  {\bf E}(G,A)$ then $\varphi$ lifts to $E$ if and only if $\varphi^*(\ms E) =0 \in {\bf E}(G',A)$. 
 \end{prop}
The group $\Hom_c(G',A)$ of continuous homomorphisms from $G'$ to $A$ acts simply transitively on the set of lifts $\hat \varphi$ of $\varphi$ by the prescription $(\chi \odot \hat \varphi)(g)=\chi(g) \hat \varphi(g)$. Therefore, when $\varphi$ lifts to $E$, the set of lifts is in bijection with $\Hom_c(G',A)$.

  \subsection{An Exact Sequence}
Let $G$ be a topological group, and put $\pi_0=\pi_0(G)$. The inclusion map $\iota$ and quotient map $\qoppa$ (``qoppa'', the Greek $q$) give an exact sequence
\begin{equation} \label{qoppa}
1 \to G^\circ \overset{\iota}{\to} G \overset{\qoppa}{\to} \pi_0 \to 1.
\end{equation}

  \begin{prop} \label{extexact} The following sequence is exact.  
  \beq
  0 \to \mb E(\pi_0,A) \overset{\qoppa^*}{\to} \mb E(G,A) \overset{\iota^*}{\to} \mb E(G^\circ,A).
  \eeq
 
  \end{prop}
  
  \begin{proof}

First we prove $\qoppa^*: \mb E(\pi_0, A) \to \mb E(G,A)$ is injective.  Let $\ms E=(E,p,i) \in {\bf E}(\pi_0, A)$, and suppose $\qoppa^* \ms E$ is trivial. Let $s': G \to E'$ be a splitting of $p'$, and write $\tilde \qoppa: E' \to E$ for the projection. 

 \beq
  	\xymatrix{ & A \ar[d] & A \ar[d]\\
& E' \ar[r]^{\widetilde{\qoppa}} \ar[d]^{p'} & E \ar[d]^p\\
G^\circ \ar[r]^{\iota}& G \ar[r]^{\qoppa} \ar@/^/[u]^{s'}  & \pi_0\\ 	
}
\eeq

Since $p \circ \tilde \qoppa  \circ s'  \circ \iota=\qoppa \circ \iota= 1$,i.e., is trivial map, the image of   
\beq
  \tilde \qoppa   \circ s'  \circ \iota: G^\circ \to E
\eeq
 lies in $A$. Since $G^\circ$ is connected, it must be trivial. So the map $\tilde \qoppa  \circ s': G \to E$  descends to a map $s: \pi_0 \to E$ which is a splitting of $p$. (Since $p \circ \tilde \qoppa \circ s'=\qoppa$.) Hence $\qoppa^*$ is injective.

 Next, we show the sequence is exact at $\mb E(G,A)$. So let $\ms E=(E,p,i) \in {\bf E}(G,A)$ with $\iota^* \ms E=(E',p',i')$ trivial. 
 We can identify $E'$ with the normal subgroup $p^{-1}(G^\circ)$ of $E$. Let $s': G^\circ \to E'$ be a splitting; its image $E^0$ is the connected component 
 $(E')^\circ$.  Hence it is normalized by conjugation from $E$. Put  $E_0=E/E^0$,
 and note that $p$ descends to $p_0: E_0 \to \pi_0$. Moreover put $i_0=\tilde \qoppa \circ i: A \to E_0$. Then $\ms E_0=(E_0,p_0,i_0) \in \mb E(\pi_0,A)$ and $\qoppa^* \ms E_0=\ms E$. 

 \end{proof}
 
When $H  \unlhd G$ is a normal subgroup, the group $G$ acts on $H$ by conjugation, and hence  on $\mb E(H,A)$. This action descends to an action of $G/H$ on $\mb E(H,A)$ by Corollary \ref{inner.triv.action}. 

\begin{lemma} \label{rest.fixed.pts} The image of the restriction map $\mb E(G,A) \to \mb E(H,A)$ is fixed by the action of $G$ on $\mb E(H,A)$.
\end{lemma}

\begin{proof} 
Let $\ms E =(E,i,p) \in {\bf E}(G,A)$ and $g \in G$. Pick $y \in E$ with $p(y)=g$. Let $E_H$ be the preimage of $H$ under $p$; it is a normal subgroup of $E$ containing the image of $i$. The restriction map takes $\ms E$ to $\ms E_H=(E_H, p|_{E_H},i)$.
It is easy to see that $\Int(y)$ covers $\Int(g)$, so $\Int(g)^*(\ms E_H) \cong \ms E_H$ by Lemma \ref{liftingtheta}.
\end{proof}

Thus Proposition \ref{extexact} may be refined to the exact sequence
  \begin{equation} \label{left.ex.here}
  0 \to \mb E(\pi_0(G),A) \to \mb E(G,A) \to \mb E(G^\circ,A)^{\pi_0(G)}.
  \end{equation}

In the next section we give a counterexample to show that this sequence  is not right exact in general, however in the subsequent section we show right exactness under a certain hypothesis on $G$.

  \subsection{The Weil group} \label{Weil.subs}

\begin{example} A classic example of a nonsplit extension is $G=W_\R$, the \emph{Weil group} of $\R$. See for instance \cite[(1.4.3)]{Tate}. The Weil group fits into an exact sequence
\beq
1 \to \mb C^\times \to W_{\R} \to \Gamma \to 1,
\eeq
where $\Gamma$ is the Galois group of $\C$ over $\R$. Briefly, the group $W_{\R}$ is a union $\C^\times \cup j \C^\times$, with multiplication  rules $j^2=-1 \in \C^\times$, and $jzj^{-1}=\ol z$.

Let $Q$ be the subgroup of $W_{\R}$ generated by $i \in \C^\times$ and $j$; it is evidently quaternion of order $8$, and contains the subgroup $\mu_2=\{ \pm 1\}$.

\begin{prop} \label{weil.prop} The restriction map $\mb E(W_\R,\Z/2\Z) \to \mb E(\C^\times,\Z/2\Z)^\Gamma$ is not surjective.
\end{prop}

\begin{proof} Let $\Sq: \C^\times \to \C^\times$ be the squaring map. Since $\Sq$ commutes with complex conjugation, this determines an extension  $\ms E_\bullet \in \mb E(\C^\times,\Z/2\Z)^\Gamma$ by Lemma \ref{liftingtheta}. Let $A = \Z/2\Z$. From the inclusions
\beq
\xymatrix{\mu_2  \ar[r] \ar[d] & G^\circ \ar[d]\\
Q \ar[r]   & G}
\eeq
we obtain the commutative diagram
\beq
\xymatrix{ \mb E(W_{\R}, A) \ar[d] \ar[r]& \mb E(\C^\times, A)  \ar[d]  \\
	\mb E(Q, A) \ar[r] & \mb E(\mu_2, A). }
	\eeq

Let us write $\ms E_2=(E_2,i_2,p_2)$ for the restriction of $\ms E_\bullet$ to $\mu_2$. Then $\ms E_2$ is nontrivial, since $E_2$ contains elements of order $4$.

It is easy to see that restriction  $\HH^1(Q,A) \to \HH^1(\mu_2,A)$ is the zero map, since each linear character of $Q$ is trivial on $\mu_2$, the derived group of $Q$.
From the description of $\HH^*(Q,A)$ in  \cite[Example 1.9, page 165]{Milgram}  or \cite[Section 3.2]{Malik.Spallone.SL}, one sees that the second cohomology group $\HH^2(Q,A)$ is contained in the subring of $\HH^*(Q,A)$ generated by $\HH^1(Q,A)$. Therefore the restriction $\HH^2(Q,A) \to \HH^2(\mu_2,A)$ is also the zero map, as claimed. In particular, the extension $\ms E_2$ is not in the image.

If there were an extension $\ms E \in \mb E(W_{\R},A)$ restricting to $\ms E_\bullet$, then its restriction to $Q$ would further restrict to $\ms E_2$, contradicting the above.
\end{proof}

\end{example}
  
   \subsection{Extensions of Semidirect Products}

   Let $H$ be a connected group, and $\Gamma$ a group acting discretely on $H$. 
We may form the semidirect product $G= H \rtimes \Gamma$.

 \begin{prop}\label{surj}  With the above notation, the restriction map
   \beq
   \mb E(G,A) \to \mb E (H,A)^{\Gamma}
  \eeq
  is a surjection.
  \end{prop}   
   
  \begin{proof} 
    Let  $\ms E_\bullet=(E_\bullet, i_\bullet,p_\bullet)  \in  \mb E (H,A)^{\Gamma}$. Suppose first that $E_\bullet$ is connected. 
    For each $\theta \in \Gamma$, there is a  lift $\tilde \theta$ of $\theta$ to $E_\bullet$ by Lemma \ref{liftingtheta}.
   In fact $\tilde \theta$ is unique:  Suppose $\tilde{\theta}_1$ and $\tilde{\theta}_2$ are lifts of $\theta$. Then we may define a continuous function $f: E_\bullet \to A$ by 
   \beq
   f(e^0)=\tilde{\theta_2}^{-1}(e^0)\tilde{\theta_2}(e^0).
   \eeq
    Since $E_\bullet$ is connected, $f$ is constant, and we deduce that $\tilde{\theta_1} = \tilde{\theta_2}$. The prescription $\theta \mapsto \tilde \theta$ is therefore a well-defined map $\Gamma \to \Aut(E_\bullet)$,
necessarily a homomorphism.  This corresponds to an action of $\Gamma$   on $E_\bullet$, and so we may form the  semidirect product $E=E_\bullet \rtimes \Gamma$. Clearly this gives an extension of $G$ by $A$  whose restriction to $H$ is  $E_\bullet$.
   
For $E_\bullet$ disconnected, write $E_\bullet^\circ$ for the connected component of $1$. Since $p_\bullet$ is an open map, the image of $E_\bullet^\circ$ is open in $H$, and hence closed. It follows that  $p_\bullet(E_\bullet^\circ)=H$. Let $A_1=A \cap E_\bullet^\circ$, and write $i_1: A_1 \hookrightarrow E_\bullet^\circ$ for the inclusion.
Then 
\beq
\ms E_\bullet':= (E_\bullet^\circ, i_1,p_\bullet|_{E_\bullet^\circ})
\eeq
 is a $\Gamma$-invariant $A_1$-extension of $G$, and $(E_\bullet^\circ)_{A} \cong E_\bullet$.

Recall that by Proposition \ref{likely.commutes}, the diagram
\beq
\xymatrix{
	\mb E(G,A_1) \ar[r]  \ar[d]& \mb E(H,A_1)^{\Gamma} \ar[d]\\
	 \mb E(G,A) \ar[r] & \mb E(H,A)^{\Gamma}\\
}
\eeq 
 commutes. Since $E_\bullet^\circ$ is connected, we have a lift $\ms E_1 \in \mb E(G,A_1)$ of $\ms E_\bullet'$, and this maps down to the required lift $\ms E$ of $\ms E_\bullet$.   \end{proof}
   

 \section{$G$-bundles} \label{G.bund.sec}

In this section we review the theory of (principal) $G$-bundles.  When $G$ is abelian, there is a Baer sum. When $G$ is discrete, $G$-bundles over a nice connected base are  classified by homomorphisms from the fundamental group into $G$.

\subsection{Definition} \label{g.bundle.defs}
 Let $G$ be a topological group, and $B$ a space. Let $X$ be a space with a right $G$-action, and $p:X \to B$ is a map. The pair $(X,p)$ is a \emph{$G$-bundle}, provided for all $b \in B$, there is a neighborhood $U \ni b$ and a homeomorphism $h_U: U \times G \to p^{-1}(U)$ so that for all $u \in U$ and $g,g' \in G$, we have $p(h_U(u,g))=u$ and $h_U(u,gg')=  h_U(u,g) \cdot g'$.

 A \emph{morphism} of $G$-bundles over $B$ is a map $f: X \to X'$ with $f(x \cdot g)=f(x) \cdot g$. Write $\Bun_B(G)$ for the category of $G$-bundles with base space $B$. Write $\bun_B(G)$ for the set of $G$-bundles modulo isomorphism.
 
Let $p:X \to B$ be a $G$-bundle. When $B'$ is another space and $f: B' \to B$ is a map, write $f^*(X) \subseteq B' \times X$ for the pullback. Projection of the pullback to $B'$ is a $G$-bundle, and projection to $X$  gives an \emph{overmap} $\tilde f: f^*(X) \to X$ which is $G$-equivariant.

\begin{lemma} \cite[Corollaire, page 94]{Bou.AT} \label{bou.pullback}  Let $(X,p)$ be a $G$-bundle over a space $B$ and $(X',p')$ be a $G$-bundle over a space $B'$. Suppose that $f: B' \to B$ and $f': X' \to X$ are continuous with $p \circ f'=f \circ p'$, and moreover that $f'$ is $G$-equivariant. Then 
\beq
	\xymatrix{
		X' \ar[r]^{f'} \ar[d]^{p'} &X \ar[d]^{p}\\
		B' \ar[r]^{f}   & B  \\
	}
	\eeq 
	is a pullback diagram.
  \end{lemma}

\begin{defn} \label{induced.bundle} Let $(X,p)$ be a $G$-bundle over $B$, let $G'$ be a $k$-group, and $\psi: G \to G'$ a homomorphism. Write $X[G']$ for the quotient of $X \times G'$ by the equivalence relation $(x,g') \sim (x \cdot g,\psi(g)^{-1}g')$,
with $x \in X$, $g \in G$ and $g' \in G'$. There is an evident right action of $G'$ on $X[G']$. 
The map taking the class of $(x,g')$ to $p(x)$ makes $X[G']$ into a $G'$-bundle over $B$, called the $G'$-bundle \emph{induced from} $(X,p)$ by $\psi$.
\end{defn}  
Write $f: X \to X[G']$ for the map over $B$, taking $x$ to the class of $(x,1_{G'})$, and satisfying
\begin{equation} \label{universal4}
f( x \cdot g)= f(x) \cdot \psi(g).
\end{equation} 
Then $f$ is universal in the sense that if $Y$ is a $G'$-bundle over $B$ and $f': X \to Y$ is a map over $B$ satisfying \eqref{universal4}, then there is a unique isomorphism $\theta: X[G'] \overset{\sim}{\to} Y$ as $G'$-bundles over $B$.
(See  \cite[Section 6.61]{Bourbaki.VD}.)

\subsection{Baer Sum of $A$-Bundles}

Let $A$ be an abelian topological group. Given $(X_1,p_1),(X_2,p_2) \in \Bun_B(A)$, let $X_3' \subseteq X_1 \times X_2$ be the fibre product:
\beq
	\xymatrix{
		X_3' \ar[r] \ar[d]^{p_3'} &X_2 \ar[d]^{p_2}\\
		X_1 \ar[r]^{p_1}   & B  \\
	}
	\eeq 
	
If we view $X_3'$ as an $A$-space under the action $(x_1,x_2) \cdot a=(x_1,x_2 \cdot a)$, then $(X_3',p_3')$ is an $A$-bundle.
	
We may also view $X_3'$ as an $A$-space under the anti-diagonal action $(x_1,x_2) \cdot a=(x_1 \cdot a,x_2 \cdot a^{-1})$. Let $X_3$ be the quotient of $X_3'$ by the anti-diagonal $A$-action. Then $X_3$ itself has an $A$-action by $(x_1,x_2) \cdot a=(x_1 \cdot a,x_2)$.
 
 If we define $p_3:X_3 \to B$ by $p_3(x_1,x_2)=p_1(x_1)$, then $(X_3,p_3) \in \Bun_B(A)$. (A common trivialization of $X_1$ and $X_2$ over $B$ will trivialize $X_3$ as well.) We call $(X_3,p_3)$ the \emph{Baer sum} of the two bundles.

 \subsection{Associated Long Exact Sequence}

  \begin{thm}  \label{homotopy.fibration}  Let $p:X \to B$ be a $G$-bundle, with $B$ path-connected. Let $x \in X$ with $p(x)=b$. Then there is a long exact sequence
  \beq
\cdots \to \pi_n(G,1_G)  \to  \pi_n(X,x) \to  \pi_n(B,b)  \to  \pi_{n-1}(G,1_G)  \to  \cdots,
\eeq
of groups, which ends with
\beq
\cdots \to \pi_1(B,b) \to \pi_0(G,1_G).
\eeq
When $X$ is path-connected, the last map is surjective. The image of $\pi_1(G)=\pi_1(G,1_G)$ in $\pi_1(X)=\pi_1(X,x)$ is contained in the center of $\pi_1(X)$.
 \end{thm}
 
 \begin{proof}
 See \cite[Section 17.11]{Steenrod.FB}, except for the last point. The group action of $G$ on $X$ induces an action of $\pi_1(G)$ on $\pi_1(X)$. 
 One checks that for loops $a,b$ in $G$ based at $1$, and loops $\gm,\sigma \in X$ based at $x$,   we have $(\gamma * \sigma) \circ (a *b)=(\gamma \circ a)*(\sigma \circ b)$, where $*$ is the usual concatenation of paths, and $\circ$ is the action. This gives $(\gamma \cdot  \sigma) \circ (a  \cdot b)=(\gamma \circ a) \cdot (\sigma \circ b)$  for $a,b \in \pi_1(G)$ and $\gm, \sigma \in \pi_1(X)$, where `$\cdot$' is group multiplication.
Now let $a \in \pi_1(G)$; its image in $\pi_1(X)$ is $1_X \circ a$, where $1_X$ is the identity in $\pi_1(X)$.
By the above we have
\beq
\begin{split}
(1_X \circ a) \cdot (\sigma \circ b) &= \sigma \circ (a \cdot b) \\
					&=\sigma \circ (b \cdot a) \\
					&=(\sigma \circ b) \cdot (1_X \circ a), \\
					\end{split}
					\eeq
 hence the image is central.
 \end{proof}

\subsection{$D$-covers}

When $D$ is a discrete group,  a $D$-bundle is necessarily a covering space, often simply called a \emph{$D$-cover}. 
 Let $B$ be path-connected, and $p:(X,x_0) \to (B,b)$ a pointed $D$-cover.  Given   $[\gm] \in \pi_1(B,b)$, let $\tilde \gm$ be the path in $X$ lifting $\gm$ with $\tilde \gm(0)=x_0$. Define $\partial_p([\gm])$ to be the unique $d \in D$ so that $x_0 \cdot d=\tilde \gm(1)$.
 
\begin{thm} \label{groth.}  \cite[Section 14a]{fulton.AT}
The above prescription defines a homomorphism $\partial_p: \pi_1(B,b) \to D$. If $B$ is nice and path-connected, then the correspondence $(X,p) \leadsto \partial_p$ gives a natural bijection
\beq
\{ \textrm{pointed $D$-covers of $(B,b)$} \}/ \text{isom} \overset{\sim}{\to}  \Hom(\pi_1(B,b),D).
\eeq
\end{thm}

Note that $\partial_p$ agrees with the connecting map for the LES of homotopy groups, as defined in \cite[page 453]{Bredon}.

\subsection{Central Extensions  as Covers.}
Let $G$ be a nice path-connected group, and $\ms E=(E,i,p) \in {\bf E}(G,A)$. 

\begin{lemma} $p:E \to G$ is an $A$-cover.
\end{lemma}

\begin{proof} 
It is enough to see that the action of $A$ on $E$ by left translation is even. Pick a neighborhood $U$ of $1_E$  so that $U \cap i(A)=\{ 1_E\}$, and a neighborhood $V$ of $1_E$  so that $V V^{-1} \subseteq U$.
It is clear that for $0_A \neq a \in A$, the intersection $V \cdot i(a) \cap V =\emptyset$.
\end{proof}

Let $\ms E=(E,i,p) \in {\bf E}(G,A)$.   Since $G$ is locally path-connected, so too is $E$, hence $E^\circ$ is open in $E$.
\begin{lemma} \label{dec.30} We have $E=E^\circ \cdot i(A)$.
\end{lemma}

\begin{proof}	
The subgroup $E^\circ$ and each of its cosets are open, hence $E'=E^\circ \cdot i(A)$ is an open subgroup of $E$. Since $p$ is an open map, the image $p(E')$ is an open subgroup of $G$.  Since $G$ is connected, we have $G=p(E')$, and therefore $E=E'$.
\end{proof}	

\begin{prop} \label{dont.forget} Let $\ms E_1,\ms E_2 \in {\bf E}(G,A)$. If the projections $p_1: E_1 \to G$ and $p_2:E_2 \to G$  are isomorphic as $A$-covers, then  $\ms E_1 \cong \ms E_2$.
\end{prop}

\begin{proof}
Let $f: E_1 \to E_2$ be an isomorphism in the category of $A$-covers of $G$. By translating by $A$ we may assume that $f(1_{E_1})=1_{E_2}$. Consider the map
\beq
E_1^\circ \times E_1^\circ \to A
\eeq
given by $(x,y) \mapsto i_2^{-1}(f(xy)^{-1}f(x)f(y))$. Since the domain is connected and the codomain is discrete, this is in fact  the constant map $(x,y) \mapsto 0_A$. 
Thus the restriction of $f$ to $E_1^\circ$ is a group homomorphism. Using that $E_1=E_1^\circ \cdot i_1(A)$ by Lemma \ref{dec.30}, we see that $f$ itself is a group homomorphism. The result follows.\
\end{proof}

\begin{thm} \label{cmi.thurs}For $G$ a nice path-connected group, the map taking an extension $\ms E =(E,i,p) \in{\bf E}(G,A)$  to the connecting homomorphism $\partial_{\ms E}$ (from the $A$-bundle $p$) is a group isomorphism
\beq
\mb E(G,A)   \overset{\sim}{\to} \Hom(\pi_1(G),A).
\eeq
\end{thm}

\begin{proof}  
Let $\psi: \pi_1(G) \to A$, and recall from Example \ref{univ.gp} the cental extension $ \ms E_G=(\tilde G,i,p)$ of $G$ by $\pi_1(G)$.  
 Then  $\psi_* \ms E_G \in \mb E(G,A)$, and its connecting homomorphism is $\psi$ by the naturality assertion of Theorem  \ref{groth.}.
 So the map is surjective. By Theorem \ref{groth.} and Proposition \ref{dont.forget}, two central extensions with the same connecting homomorphism are isomorphic. Hence the map $\ms E \mapsto \partial_{\ms E}$ is a bijection.
\end{proof}

\begin{remark} \label{remarkable} Under this bijection, the extension $(\tilde G,i,p)$ of Example \ref{univ.gp} corresponds to the identity map in $\Hom(\pi_1(G),\pi_1(G))$.
\end{remark}

 \section{Topological Group Theory Preliminaries} \label{class.space.section}
 
 In this section we recall the general notion of a classifying space, and some technical bits of topology.
 
 \subsection{Universal $G$-bundles} \label{BG.section}

 \begin{defn} A $G$-bundle   $p:E \to B$ is \emph{universal}, provided $E$ is contractible, and $B$ is paracompact. In this case, we call $B$ a \emph{classifying space} for $G$.
\end{defn}

\begin{example} If $G=\Z$, then $\exp: \C \to \C^\times$ is a universal $\Z$-bundle. So $\C^\times$ is a classifying space for $\Z$.
\end{example}

Let $E_u \to B_u$ be a universal $G$-bundle. When $B$ is paracompact,  the prescription
\beq
f \leadsto (f^*(E_u) \to B)
\eeq
gives a bijection 
 \begin{equation} \label{univ.class.propz}
[B,B_u] \overset{\sim}{\to} \bun_B(G).
\end{equation}
Here  $[X,Y]$ is the set of homotopy classes of maps from $X$ to $Y$. See, for instance \cite[page 45]{Milgram} or \cite[Chapter 4, Section 10]{Husemoller}.  

 \begin{lemma} \label{bu.he} Any two classifying spaces for $G$ are homotopically equivalent.
 \end{lemma}
 
 \begin{proof} Suppose $p: E \to B$ and $p': E' \to B'$ are both universal $G$-bundles. Let $f: B \to B'$ so that $E$ is the pullback of $E'$, and $f': B' \to B$ so that $E'$ is the pullback of $E$. Then $E$ is the pullback of $E$ under $f' \circ f$, and therefore $f' \circ f$ is homotopic to the identity map on $B$. Similarly, $f \circ f'$ is homotopic to the identity map on $B'$, and so $B$ and $B'$ are homotopy equivalent.
 \end{proof}
 
 \subsection{Kelley Groups}
 
 
 
   \begin{defn}
Let $B$ be a Hausdorff space and $b \in B$. We say the pair $(B,b)$ is \emph{well-pointed} when the inclusion  $\{ b\} \hookrightarrow B$ is a cofibration \cite[Chapter VII]{Bredon}. 
\end{defn}
 For example CW-complexes are well-pointed, but the Cantor set is not.  A cover of a well-pointed space is also well-pointed.
 
 \bigskip
 
 Subtle topological issues arise  for products of spaces, when neither is locally compact. For example, if $X$ and $Y$ are infinite-dimensional CW-complexes, then the product set $X \times Y$ with the usual product topology may not be a CW-complex. Topologists prefer to work with ``compactly generated'' spaces, which include CW-complexes. 
 
 \begin{defn} A Hausdorff space $X$ is \emph{compactly generated} (CG), provided a subset of $X$ is closed iff its intersection with every compact subset of $X$ is closed.
 \end{defn}
 
Unfortunately, when $X$ and $Y$ are CG spaces, then $X \times Y$ with the  product topology may not be CG. 
 There is a standard refinement of the product topology, sometimes called the \emph{$k$-topology} on $X \times Y$, which remedies this, in which the closed sets are those whose intersection with every compact set is closed, when considered in the product topology. We will write this as ``$X \times_k Y$''. When either $X$ or $Y$ is locally compact, then the topologies coincide, i.e., $X \times_k Y=X \times Y$. Please see \cite[Chapter VII, Section 8]{Maclane} for a textbook account.
 
 Let us draw out this distinction for groups with a topology. (See \cite{Lamartin}.)  
 
 \begin{defn} Let $G$ be a group with a CG topology, and continuous inversion. Write $\mu: G \times G \to G$ for the multiplication map. When $\mu$ is continuous for the $k$-topology on $G \times G$, we say $G$ is a \emph{$k$-group} (Kelley group). When $\mu$ is continuous for the usual product topology on $G \times G$ we say $G$ is a \emph{$t$-group} (CG topological group). 
 \end{defn}

Every $t$-group is a $k$-group, but not  every  $k$-group is a $t$-group. The continuity condition for a \emph{right $G$-action} of a $k$-group $G$ on a space $X$ is that it defines a continuous map $X \times_k G \to X$.  
 To define a $G$-bundle, one must take $k$-products in Section \ref{g.bundle.defs}.

  \begin{lemma}
A cover of a regular CG-space $B$ is again a regular CG-space.
 \end{lemma}
 
 \begin{proof}
 Let $p: X \to B$ be the cover.  One can show that for all $x_0 \in X$, there is an open neighborhood $V$ of $x_0$ so that  $p$ is trivial over $p(\ol V)$, and the restriction fo $p$ to $\ol V$ is a closed embedding.
   Let $U$ be a $k$-open subset in $X$, and $x_0 \in U$. Taking $V$ as above, it follows  that $p(U \cap V)$ is $k$-open in $B$ and hence open.  From this, $U \cap V$ is open in $X$ and ranging over all $x_0$ we deduce that $U$ is itself open. 
  This shows that $X$ is a CG-space, and regularity is straightforward.
 \end{proof}
 
 It follows that a central extension of a well-pointed $t$-group is again a well-pointed $t$-group. (Recall every topological group is regular.)
 
 \section{Milgram-Steenrod's Classifying Space} \label{Milgram.Steenrod.sec}
  Now we review the Milgram-Steenrod  $BG$-functor and its properties.  Applying this functor to $\ms E \in {\bf E}(G,A)$ gives a $BA$-bundle over $BG$ which we will use heavily. We also exhibit a bijection from $\mb E(G,A)$ to $\HH^2(BG,A)$ when $G$ is connected, by using the connecting map of the LES of homotopy groups. Finally, we study the structure of the $K(A,2)$-space $B(BA)$.
  
  \subsection{$EG$ as a $G$-space}
 
  Given a well-pointed $k$-group $G$, Steenrod in \cite{Steenrod} recursively constructs a sequence
 \beq
 D_0 \subseteq E_0 \subseteq D_1 \subseteq E_1 \subseteq \cdots
 \eeq
 of spaces, so that each $E_i$ has a free right $G$-action, and each $D_n$ comes with a specific contraction.
 
 $D_0=\{e\}$, $E_0=G$, and
 $D_1$ is the reduced cone $G \wedge I$, where $I$ is the unit interval $[0,1]$. Here the \emph{reduced cone} is the topological quotient of $G \times I$ by the subspace
 \beq
 (G \times \{ 0\}) \cup ( \{1_G\} \times I).
 \eeq
 Generally, $D_{n+1}$ is the ``enlargement to $E_{n}$ of the contraction of $D_n$'', meaning it is defined as the pushout

 \beq
	\xymatrix{
		D_n \wedge I \ar[r] \ar[d]  &D_n \ar[d] \\
		E_n \wedge I \ar[r]   & D_{n+1}  \\
	}
	\eeq	
	
	and $E_{n+1}$ is the ``enlargement to $D_{n+1}$ of the $G$-action on $E_{n}$'', defined as the pushout
	 
 \beq
	\xymatrix{
		  E_n \times G \ar[r] \ar[d]  &E_n \ar[d] \\
		  D_{n+1} \times G \ar[r]   & E_{n+1}  \\
	}.
	\eeq 
	
 Then $EG$ is the union of the $E_n$'s, equivalently the union of the $D_n$'s. We give it the weak topology of the $D_n$'s. From being the union of the $E_n$'s, it inherits a free right $G$-action. From being the union of the $D_n$'s, it inherits a contraction.
 
 Finally, the quotient of $EG$ by $G$ defines a universal $G$-bundle $p=p_G:EG \to BG$. Write $B_nG \subset BG$ for the image $p(D_n)=p(E_n)$.

When $G$ is discrete, the spaces $EG$ and $BG$ are CW-complexes, with $E_nG$ and $B_nG$  the $n$-skeletons of $EG$ and $BG$, respectively.
 
\begin{remark} More generally, when $G$  is a CW-complex with a cellular multiplication, then $EG$ and $BG$ are also $CW$-complexes. 
The multiplication on $EG$ is again cellular, and similarly for $BG$ when $G$ is abelian. For example, this holds for  $G=\SO(n)$ \cite[Section 3.D]{hatcher}.
\end{remark} 
 
 \begin{example} When $G$ is trivial, $D_n=E_n$ is a point for all $n$. So, $EG=BG$ is a point.
 \end{example}
  
  \begin{example} When $G=C_2$,  
  $D_n$ is the $n$-disc, and $E_n=S^n$ with antipodal action of $G$. So, $EG=S^\infty$ and $BG=\R \mb P^\infty$,
  with $B_nG=\R \mb P^n$.
  \end{example} 
  
  \subsection{$EG$ as a group} \label{EG.gp}
  
Steenrod defines in \cite{Steenrod} a group structure on $EG$, making it a well-pointed $k$-group. We record some of its properties here.  
\begin{itemize}
\item $EG$ is generated as a group by $D_1$. 
\item The identity $G =E_0$ gives an injective homomorphism $i: G \to EG$.
\item The $G$-action on each $E_n$ agrees with right multiplication by $E_0$.
\end{itemize}
  
  Since $EG$ is generated by $D_1$, the set-map 
  \beq
  G \times I \twoheadrightarrow G \wedge I= D_1
  \eeq
  induces a surjection $F(G \times I) \twoheadrightarrow EG$,
  where $F(G \times I)$ is the free group on the set $G \times I$. From this surjection, we can identify the group $EG$   with the quotient of $F(G \times I)$ by the following relations
 \begin{itemize}
 \item $(g,0)=(1_G,t)=1$ for $g \in G$ and $t \in I$.
 \item $(g,t)(g',t)=(gg',t)$ for $g,g' \in G$ and $t \in I$.
 \item If $0<t'<t \leq 1$, then $(g,t)(g',t')=(gg'g^{-1},t')(g,t)$ for $g,g' \in G$.
 \end{itemize}
 For simplicity we will again write $(g,t) \in EG$, but put $[g,t]=p((g,t)) \in BG$. With this notation, the map $i: G \to EG$ is the homomorphism defined by $i(g)=(g,1)$. Then $BG$ is the quotient $EG/i(G)$. Let $e_{BG}=p(1_{EG}) \in BG$. When $G$ is abelian, then $EG$ is abelian, and so $BG$ is an abelian well-pointed $k$-group.

We may write
\beq
D_n=\{ (g_1,t_1) \cdots (g_n,t_n) \mid g_i \in G, t_i \in I\}.
\eeq
Write $[g_1,t_1] \cdots [g_n,t_n]:=p \left( (g_1,t_1) \cdots (g_n,t_n) \right)$, so that $B_nG=\{ [g_1,t_1] \cdots [g_n,t_n] \mid g_i \in G, t_i \in I\}$.  

Recall that for a pointed topological space $(X,x_0)$, one writes $\Sigma X$ for the the \emph{reduced suspension} of $X$, meaning the quotient of $X \times I$ by 
 \beq
 \left( X \times \{0\} \right) \cup \left( X \times \{1\} \right) \cup \left(
 \{ x_0\} \times I \right).
 \eeq
 For example, $B_1G \cong \Sigma G$.
 We will   use notation $(x,s_1,s_2) \in \Sigma^2 X$ with $x \in X$ and $s_1,s_2 \in I$, to mean the image under the evident surjection $X \times I^2 \to \Sigma^2 X$. Note that this surjection factors through $X \times I^2/\partial I^2 \to \Sigma^2 X$.
 
 Write $\Delta_n$ for the set of $n$-tuples $(t_1, \ldots, t_n)$ with $0 \leq t_1 \leq \cdots \leq t_n \leq 1$. Let $*$ be the basepoint of $ I^2/\partial I^2$. The map $\Delta_2 \to I^2/\partial I^2$ defined by
 \beq
 (t_1,t_2) \mapsto \begin{cases} 
 & (t_1/t_2,t_2) \text{ when $t_2 \neq 0$}\\
 & * \text{ when $t_2=0$}\\
 \end{cases}
 \eeq
 descends to a homeomorphism $\eta_2:\Delta_2/\partial \Delta_2 \overset{\sim}{\to} I^2/\partial I^2$. More generally, from the formula $(t_1, t_2, \ldots, t_k) \mapsto (t_1/t_2,t_2/t_3, \ldots, t_k)$ one obtains a homeomorphism
 $\eta_k:\Delta_k/\partial \Delta_k \overset{\sim}{\to} I^k/\partial I^k$, with inverse $\bar \eta_k$ given by
 \beq
 \tilde \eta_k(x_1, \ldots, x_k)=(x_1x_2 \ldots x_k,\ldots, x_{k-1}x_k,x_k).
 \eeq
 Steenrod defines quotient maps $k_n: G^n \times \Delta_n \to D_n$. The map $q: G^2 \times \Delta_2 \to \Sigma^2(G \wedge G)$ given by
 \beq
 (g_1,g_2,t_1,t_2) \mapsto \left[ \left(g_1,g_2,\eta_2(t_1,t_2) \right) \right]
 \eeq
 factors through $p \circ k_2$, and induces a homeomorphism
 \begin{equation} \label{suspend.disbelief}
 q: B_2G/B_1G \overset{\sim}{\to} \Sigma^2(G \wedge G),
 \end{equation}
 with $[g_1,s_1][g_2,s_2] \mapsto ((g_1,g_2),s_1,s_2)$.

Given a continuous group homomorphism $\varphi: G \to G'$, the map $G \times I \to G' \times I$ defined by $(g,t) \mapsto (\varphi(g),t)$
induces a continuous group homomorphism $E\varphi: EG \to EG'$
with the property that $E\varphi(x \cdot g)=E\varphi(x) \cdot \varphi(g)$. Hence it descends to a continuous map $B\varphi: BG \to BG'$.

  \begin{prop}  \label{fibration.normal}  \cite[Theorem 7.7]{Piccinini} Let  $G$ be a well-pointed $t$-group. If $\mathscr E=(E,i,p) \in  {\bf E}(G,A)$, then the induced map $Bp: BE \to BG$ is a $BA$-bundle. 
  \end{prop}
  
\begin{remark} The group $E$ is a well-pointed $t$-group, since it covers the well-pointed $t$-group $G$.    So via \cite[Lemma 7.4]{Piccinini},  the hypothesis of \cite[Theorem 7.7]{Piccinini} is satisfied. 
\end{remark}
 
 \subsection{Iterating $B$} \label{B^2}
 
 Since $A$ is abelian, so is $EA$. Hence the coset space $BA=EA/A$ is also a group, generated by quotienting $A \times I$ by the further relation $(a,1) \equiv 0_{BA}$ for $a \in A$. Applying the functor $B$ again, we see that $BBA$ is generated by the elements $[[a,t],t']$ for $a \in A$ and $t,t' \in I$. For simplicity we drop the inner brackets, writing $[[a,t],t']$ as $[a;t,t']$, so that $BBA$ is the abelian group generated by $A \times I^2$ with the following relations for $a,a_i \in A$ and $t,t' \in I$:
 \beq
 (a_1;t,t')+(a_2;t,t') \equiv (a_1+a_2;t,t'),
 \eeq
 and $(a;t,t') \equiv 0$ iff one of the following holds:
 \begin{itemize}
 \item $a=0_A$
 \item $t=0$ or $t'=0$
 \item $t'=1$ or $t'=1$.
 \end{itemize}
 Thus $[a;t_1,t_2]$ is the image of $(a;t,t')$ in the quotient $B^2A$.
 This generalizes in the obvious way for $B^nA$.
 
 \begin{remark} Write $\C(x)$ for the field of rational functions with complex coefficients, and $\mb P(\C(x))$ for the multiplicative quotient group $\C(x)^{\times}/\C^\times$. Let $\varphi$ be a homeomorphism from $I^2/\partial I^2$ to the Riemann sphere, taking the boundary  to the infinite point.
 
The following prescription describes a group isomorphism from $B^2\Z$ to  $\mb P(\C(x))$. Take a generator $[n;t,t']$, with $(t,t') \notin \partial I^2$, to the power $(z-\varphi(t,t'))^n$, and when  $(t,t') \in \partial I^2$, map $[n;t,t']$ to the identity. \end{remark}
 
 \begin{lemma} \label{gamma.a} Given $a \in A$, the map $I^n \to B^nA$ given by $(t_1,\ldots, t_n) \mapsto [a; t_1, \ldots, t_n]$ descends to $\gm^n_a: I^n/\partial I^n \to B^nA$.
 The correspondence $a \mapsto \gm^n_a$ gives an isomorphism $A \overset{\sim}{\to} \pi_n(B^nA)$.
 \end{lemma}
 
 \begin{proof} For each $n$ one has a $B^{n-1}A$-bundle $E(B^{n-1}A)\to B^nA$, whence connecting isomorphisms $\pi_n(B^nA) \overset{\partial}{\to} \pi_{n-1}(B^{n-1}A)$. Now $\gm^n_a$ lifts to $\tilde \gm^n_a: I^n \to E(B^{n-1}A)$ given by $\tilde \gm^n_a(t_1, \ldots, t_n)=([a;t_1, \ldots, t_{n-1}],t_n)$. 
 The restriction of  $\tilde \gm^n_a$ to all but one of the faces of $I^n$ is the constant map to $1$, but the restriction of $\tilde \gm^n_a$ to the face of $I^n$ with $t_n=1$ of $I^n$ can be identified with $\gm^{n-1}_a$. Hence $\partial(\gm^n_a)=\gm^{n-1}_a$.  Composing them gives an isomorphism
 \beq
 \pi_n(B^nA) \overset{\partial}{\to} \pi_{n-1}(B^{n-1}A)  \overset{\partial}{\to} \cdots  \overset{\partial}{\to} \pi_1(BA) \overset{\partial}{\to} A
 \eeq
 sending $\gm^n_a$ to $a \in A$.
 \end{proof}

 \subsection{From Extensions to Bundles} \label{rain.car}
 Let $G$ be a well-pointed $t$-group, and  $\ms E=(E,i,p) \in  {\bf E}(G,A)$. As noted above (Proposition \ref{fibration.normal}), $Bp: BE \to BG$ is a $BA$-bundle, and we write $\ms B \ms E=(BE,Bp)$. 
   It is easy to see that an isomorphism $\ms E_1 \cong \ms E_2$ induces a $BA$-bundle isomorphism from $\ms B \ms E_1$ to $\ms B \ms E_2$.
  
So we have a functor $ \ms B: {\bf E}(G,A) \to  \Bun_{BG}(BA)$ which induces a well-defined map $\ms B: \mb E(G,A) \to \bun_{BG}(BA)$.
Suppose $G,G'$ are well-pointed $t$-groups, $\ms E \in {\bf E}(G,A)$ and    $\ms E' \in {\bf E}(G',A)$, and  $f: E \to E'$ is a homomorphism with $f \circ i=i' $. Then one has a commutative diagram
 \beq 
	\xymatrix{
		A \ar[d]^{\id} \ar[r]^i & E  \ar[d]^{f}\ar[r]^p & G \ar[d] \\
		A \ar[r]^{i'}  & E' \ar[r]^{p'} & G'.}
	\eeq 
    
   Applying $B$ to the group action diagram 
   
    \beq 
	\xymatrix{
		E \times A \ar[r] \ar[d]_{f \times \id_A} & E  \ar[d]^{f}\\
		E' \times A \ar[r]^{ } & E',}
	\eeq  
 shows that $Bf$ is a $BA$-equivariant morphism from $BE$ to $BE'$.

\begin{prop} \label{Bp.funct}  If $\varphi: G' \to G$ is a homomorphism, and  $\ms E \in {\bf E}(G,A)$ then 
\beq
\ms B(\varphi^* \ms E) \cong (B \varphi)^* (\ms B \ms E).
\eeq
\end{prop}
 
 \begin{proof}
Applying  $\ms B$  to the pullback square for $\varphi^*E$ gives the commutative diagram 
 \beq 
	\xymatrix{
		B (\varphi^*E) \ar[r] \ar[d] & BE  \ar[d]^{Bp} \\
		BG' \ar[r]^{B \varphi} & BG}.
	\eeq  

We obtain the desired isomorphism by Lemma \ref{bou.pullback}.	

  \end{proof}
  
  \begin{prop} \label{Baer.2.Baer} The $\ms B$-functor takes Baer sums to Baer sums.
  \end{prop}

  \begin{proof} Let $\ms E_1,\ms E_2 \in{\bf E}(G,A)$ as in Section \ref{Baer.groups}. Note that the projection $E_3' \to E_2$ gives an extension  $\ms E_3'=(E_3,i_3',p_3') \in {\bf E}(E_2,A)$, and therefore $BE_3' \to BE_2$ is a $BA$-bundle.
  Both $BE_3'$ and $BE_1 \times_{BG} BE_2$ have two actions of $BA$. 
  The first action of $BA$ on $BE_3'$ comes from $Bi_3'$, and the first action on the fibre product comes from the $BA$-action on $BE_1$. The second action on $BE_3'$ comes from $B\triangledown$, and the second action on the fibre product comes from the antidiagonal action of $BA$ on the product.
   
Consider the function
   \beq
 f=B\pr_1 \times B\pr_2: BE_3' \to BE_1 \times_{BG} BE_2.
 \eeq
 One checks that $f$ is $BA$-equivariant with respect to both actions. 
 From the first equivariance, it induces an isomorphism of $BA$-bundles over $BE_2$. 
 From the second equivariance, we see that $f$ descends to an isomorphism from $\ms B \ms E_3$ to the Baer sum
 $\ms B \ms E_1 + \ms B \ms E_2$.
  \end{proof}

 \begin{prop} $\ms B$ is natural in $A$.
 \end{prop}
 
 \begin{proof} Let $\psi: A \to A'$, and $\ms E \in {\bf E}(G,A)$. 
 Let $E'$ be as in \eqref{e.prime}. From Diagram \eqref{diag.ek}, we obtain a map $f': BE \to BE'$ over $BG$, satisfying $f'(x \cdot a)=f'(x) \cdot B \psi(a)$
 for all $x \in BE$ and $a \in BA$.  Thus as in Section \ref{g.bundle.defs}
 we deduce an isomorphism
 \beq
 \theta: BE[BA'] \overset{\sim}{\to} BE'
 \eeq
 of $BA'$-bundles over $BG$.
 \end{proof}

\section{Homotopy Theory of Classifying Spaces} \label{long.ex.seq.section}
 In this section let $G$ be a nice well-pointed $t$-group, for example a Lie group.
\subsection{The Long Exact Sequence}
Applying Theorem \ref{homotopy.fibration} to the   $G$-bundle $EG \to BG$ gives a  long exact sequence
\beq
\cdots \to \pi_n(G)  \to  \pi_n(EG) \to  \pi_n(BG)  \to  \pi_{n-1}(G)  \to  \cdots,
\eeq
which is natural in $G$. Since $EG$ is contractible, this gives isomorphisms 
\begin{equation} \label{bottle}
\partial_n=\partial_n^G: \pi_n(BG) \overset{\sim}{\to} \pi_{n-1}(G),
\end{equation} 
for $n \geq 1$.   The set of path components $\pi_0(G,1)=G/G^\circ$ has a group structure, and 
\begin{equation}\label{pi1bg}
\partial_1: \pi_1(BG) \to \pi_0(G,1)
\end{equation}
is an isomorphism. By Hurewicz's theorem, this gives rise to natural isomorphisms
 \begin{equation} \label{stapler} 
 \partial_1: \HH_1(BG) \overset{\sim}{\to}  (G^\circ \bks G)^{\ab},
  \end{equation}
 and
\beq
\partial_1^*: \Hom(\HH_1(BG),A) \to \Hom_c(G,A).
\eeq
 Since $A$ is discrete, 
\beq
\Hom_c(G,A)=\Hom((G^\circ \bks G)^{\ab},A).
\eeq   
 By the UCT, we have
\beq
0 \to \Ext^1(\HH_0(BG),A) \to \HH^1(BG,A) \overset{q^1}{\to} \Hom(\HH_1(BG),A) \to 0.
\eeq
Since $H_0(BG)$ is free, the $\Ext$-term vanishes, and so $q^1$ gives an isomorphism from $H^1(BG,A)$ to $\Hom(\HH_1(BG),A)$. 
Combining this with   \eqref{stapler} gives:

\begin{prop} \label{m^1}
The composition 
\beq
 m^1= \partial_1^* \circ q^1:\HH^1(BG,A) \to \Hom_c(G,A)
 \eeq
  is a natural isomorphism.
\end{prop}
     
The Hurewicz map $h_2: \pi_2(BG) \to \HH_2(BG)$ dualizes to a homomorphism
\beq
h^2=h_2^*: \Hom(\HH_2(BG),A) \to \Hom(\pi_2(BG),A).
\eeq
Moreover the Universal Coefficient Theorem gives an exact sequence
\begin{equation} \label{UCT}
0 \to \Ext^1(\HH_1(BG), A) \to \HH^2(BG,A) \overset{q^2}{\to} \Hom_{\Z}(\HH_2(BG),A) \to 0.
\end{equation}
 
The map $m^2: \HH^2(BG,A) \to \Hom(\pi_1(G),A)$ defined by $m^2=(\partial_2^*)^{-1} \circ h^2 \circ q^2$ is the composition of natural maps and hence is natural itself.
 
 \subsection{Case of $G$ connected}
 \begin{prop} \label{Gconnect.m2} When $G$ is connected, $m^2$ is a bijection.
 \end{prop}
 
 \begin{proof} 
 By \eqref{bottle}, the group $\pi_1(BG)$ is trivial, and so necessarily $\HH_1(BG)$ is trivial. It then follows that  $h_2$ is an isomorphism. The $\Ext$-term in \eqref{UCT} vanishes, and so $m^2$ is the composition of two bijections.
   \end{proof}
   
 Combining $m^2$ with the bijection of Theorem \ref{cmi.thurs} gives:
  
 \begin{thm}\label{Gconnect.equiv} Let $G$ be a nice well-pointed connected $t$-group. The above maps give  an isomorphism
\beq
\HH^2(BG,A) \overset{\sim}{\to} \Hom(\pi_1(G),A) \overset{\sim}{\to} \mb E(G,A).
\eeq
\end{thm}
\begin{proof}
The first equivalence follows from Proposition \ref{Gconnect.m2}.
\end{proof}

\begin{defn} \label{beta.defn} For a nice well-pointed connected $t$-group $G$, write
\beq
\beta_G: \mb E(G,A)  \to \HH^2(BG,A)  
\eeq
for the inverse of the above isomorphism.
\end{defn}

Note that $\beta_G(\ms E)$ is characterized by the property that
\begin{equation} \label{characterized} 
m^2(\beta_G(\ms E))=\partial_{\ms E},
\end{equation}
where $\partial_{\ms E}$ is the connecting homomorphism of Theorem \ref{cmi.thurs}.

 \subsection{Case of $G$ discrete}
 
Now let $G$ be a discrete group.  If $p: X \to BG$ is a $BA$-bundle, then applying Theorem \ref{homotopy.fibration} gives 
 a central extension
 \beq
 0 \to A \to \pi_1(X) \to G \to 1,
 \eeq
since $\pi_2(BG)=\pi_1(G)=0$ and $BA$ is connected.
 
This gives a map $\ms P: \bun_{BA}(BG) \to \mb E(G,A)$. 
If $(E,p,i) \in \mb E(G,A)$, then according to  \cite[1B.9, page 90]{hatcher}, we have $\pi_1(Bp)=p$ and $\pi_1(Bi)=i$.
 Therefore $\ms P \circ \ms B$ is the identity on $\mb E(G,A)$, and in particular:

 \begin{prop} \label{B.map.inj} When $G$ is discrete, the map $\ms B$ is injective.
 \end{prop}

 \subsection{Connected Component}
 
 Let $G$ be a nice well-pointed $t$-group, and put $\pi_0=\pi_0(G)$. The inclusion map $\iota$ and quotient map $\qoppa$ give an exact sequence
\beq
1 \to G^\circ \overset{\iota}{\to} G \overset{\qoppa}{\to} \pi_0 \to 1.
\eeq
 
\begin{lemma} \label{isom.pi.bs} The map $\qoppa$   induces an isomorphism $B\qoppa_*: \pi_1(BG) \overset{\sim}{\to} \pi_1(B\pi_0)$.
   
\end{lemma}

 \begin{proof}  
 Generally, a morphism $\varphi: G \to H$ induces a commuting square
 \beq
	\xymatrix{
		\pi_{n+1}(BG) \ar[r] \ar[d]  &  \pi_{n+1}(BH) \ar[d]   \\
		\pi_n(G) \ar[r]  & \pi_n(H),   \\
	}
	\eeq
 where the vertical arrows are isomorphisms. In the case of $\qoppa$ and $n=0$, the bottom map is clearly an isomorphism, hence so too is the top map.
 \end{proof}

 \subsection{$K(A,2)$-Spaces} \label{ka2}  
   When a path-connected topological space $K$ has $\pi_2(K) \cong A$ as its only nontrivial homotopy group, it is called a $K(A,2)$-space.   
 Fixing an isomorphism $\zeta: \pi_2(K)  \overset{\sim}{\to} A$, and composing with the Hurewicz map gives an isomorphism $h_2 \circ \zeta^{-1}: A \overset{\sim}{\to} \HH_2(K)$. The inverse of this map, we can write as  $\zeta \circ h_2^{-1} \in \Hom_{\Z}(\HH_2(K),A)$. The UCT gives an isomorphism
\beq
q^2: \HH^2(K,A)\overset{\sim}{\to} \Hom(\HH_2(K),A).
\eeq
Therefore there is a unique $\iota=\iota_K \in \HH^2(K,A)$ with $q^2(\iota)=\zeta \circ h_2^{-1}$.
 
Let $f : B \to K$, with $B$ path-connected. From $f$ we obtain $\pi_2(f): \pi_2(B) \to \pi_2(K)$, 
its transpose 
\beq
\pi_2(f)^*:  \Hom(\pi_2(B),A) \to \Hom(\pi_2(K),A),
\eeq
and also $f^*: \HH^2(K,A) \to \HH^2(B,A)$.
From the naturality of $h^2$ and $q^2$ we note:
	\begin{equation} \label{m2.f.iota}
	h^2 (q^2(f^*(\iota))) =\pi_2(f)^*(\zeta). 
	\end{equation}

\begin{example} \label{bba} The classifying space $BA$ is an abelian $k$-group, as noted in Section \ref{EG.gp}. Let $K=B(BA)$; it is a $K(A,2)$-space. We may define $\zeta=\zeta_A: \pi_2(K)  \overset{\sim}{\to} A$ as the composition
\beq
\pi_2(K)  \overset{\sim}{\to}  \pi_1(BA)  \overset{\sim}{\to}  \pi_0(A)=A
\eeq
of the connecting maps from the long exact homotopy sequences associated to the $A$-bundle $EA \to BA$ and the $BA$-bundle $E(BA) \to K$.
\end{example}

Let $\psi: A \to A'$ be a homomorphism of abelian groups. It induces 
\beq
\psi_*: \HH^2(B^2A,A) \to \HH^2(B^2A,A')
\eeq
and
\beq
(B^2\psi)^*: \HH^2(B^2A',A') \to \HH^2(B^2A,A').
\eeq

\begin{lemma} \label{roti.sabzi} We have $\psi_*(\iota_A)=(B^2 \psi)^*(\iota_{A'})$.
\end{lemma} 
 
 \begin{proof}
 From the diagram
 \beq
	\xymatrix{
		\pi_2(B^2A) \ar[r] \ar[d]^{B^2 \psi_*} &  \pi_1(BA) \ar[d]^{B^2 \psi_*} \ar[r]& A \ar[d]^{\psi_*}  \\
		\pi_2(B^2A') \ar[r] & \pi_1(BA') \ar[r]  &  A'   \\
	}
	\eeq
 we see that $ \psi_* \circ \zeta_A= \zeta_{A'} \circ B^2 \psi_*$. The lemma then follows from the commutativity of the following diagram.
  \beq
	\xymatrix{
\HH^2(B^2A,A) \ar[r]^{\psi_*} \ar[d]^{h^2 \circ q^2} & \HH^2(B^2A,A')   \ar[d]^{h^2 \circ q^2}&\HH^2(B^2A',A') \ar[d]^{h^2 \circ q^2} \ar[l]_{(B^2 \psi)^*}\\
	 \Hom(\pi_2(B^2A),A) \ar[r]^{\psi_*}  &  \Hom(\pi_2(B^2A),A') &  \Hom(\pi_2(B^2A'),A') \ar[l]_{(B^2 \psi)^*} \\
	}
	\eeq
 \end{proof}

 Note that $K=BBA$ is an abelian group under 
 \beq
 \mu=B^2m: K \times K \to K,
 \eeq
 where $m$ is the addition map for $A$.
 Given $f_1,f_2: B \to K$, write
 \beq
 f_1+f_2=\mu \circ (f_1 \times f_2).
 \eeq
 
Denote by $i_1, i_2: K \to K\times K$   the inclusions $i_1(x)= (x,1_K)$ and $i_2(x)= (1_K,x)$.
 Let $\pr_1$ and $\pr_2$ be the first and second projections from $K \times K$ to $K$. 
 
 \begin{lemma}  \label{Hspace}  We have
 \beq
 (f_1+f_2)^*(\iota_A)=f_1^*(\iota_A)+f_2^*(\iota_A).
 \eeq
 \end{lemma}

 \begin{proof}

Of course,
$\mu \circ i_1=\mu \circ i_2 = \id_K.$ Thus   $i_1^*(\mu^*(x))= x $ and $i_2^*(\mu^*(x))= x $  for every $x \in \HH^2(K,A)$.
Since $K$ is simply connected, we have $\mu^*(\iota_A)= \pr_1^*(\iota_A) + \pr_2^*(\iota_A)$. Therefore
 \beq
 \begin{split}
(f_1+f_2)^*(\iota_A) &= (f_1 \times f_2)^*(\mu^*(\iota_A)) \\
	&=(f_1 \times f_2)^*(\pr_1^*(\iota_A) +\pr_2^*(\iota_A))\\
	&= f_1^*(\iota_A) + f_2^*(\iota_A). \\
\end{split}
\eeq
\end{proof}

\section{An Anticommuting Square}

\subsection{Several Bundles in One} \label{x.e.sec}
Let $A$ be a discrete group, $G$ a nice   well-pointed $t$-group, and  $\ms E=(E,i,p) \in {\bf E}(G,A)$.  We form the fibre product $X_{\ms E}=EG \times_{BG} BE$ as under: 

 \beq
	\xymatrix{
		X_{\ms E} \ar[r] \ar[d]  &  BE \ar[d]^{Bp}   \\
		EG \ar[r]^{p_G}  & BG.  \\
	}
	\eeq
	
	
 The composition either way gives a ($G \times BA$)-bundle $X_{\ms E} \to BG$. From $Bp \circ p_E=p_G \circ Ep: EE \to BG$, we deduce a map $q:Ep \times p_E: EE \to X_{\ms E}$. Since $EE$ is contractible, this is a universal $A$-bundle. In particular $X_{\ms E}$ is isomorphic to $EE/A$; this gives $X_{\ms E}$ the structure of a well-pointed $k$-group.

The fibre inclusion $G \times BA \hookrightarrow X_{\ms E}$ is a group homomorphism. Write $\iota_G: G \to X_{\ms E}$ and  $\iota_{BA}: BA \to X_{\ms E}$  for the restrictions of this inclusion to $G$ and $BA$. The following is straightforward, but shows the utility of $X_{\ms E}$:
 
 \begin{prop} \label{they're.class}
 \begin{enumerate}
\item  The inclusion $\iota_G: G \to X_{\ms E}$  is a classifying map for the $A$-bundle $(E,p)$.
 \item The inclusion $\iota_{BA}: BA \to X_{\ms E}$  is a classifying map for the $A$-bundle $(EA, p_A)$.
 \end{enumerate}
 \end{prop}
 
 Applying $B$ to $q$ gives  a universal $BA$-bundle
\beq
B q: BEE \to BX_{\ms E}.
\eeq
The pullback  $\iota_{BA}^*EE$ to $BA$ equals $EA$, hence $\iota_{BA}$ is a homotopy equivalence. Since it is also a group homomorphism, we see that $B\iota_{BA}: BBA \to BX_{\ms E}$ is a weak homotopy equivalence.

 Let $f_{\ms E}: BG \to B^2A$ be a classifying map for $Bp: BE \to BG$. Since $B \iota_G$ and 
 $(B\iota_{BA})_*f_{\ms E}$ are classifying maps pulling $BEE$ back to $BE$, they are homotopy equivalent.  
 So we may write
 \begin{equation} \label{pascal}
 (B\iota_{BA})_*f_{\ms E}=B \iota_G
 \end{equation}
 in $[BG,BX_{\ms E}]$. Since $(B\iota_{BA})_*: [BG,B^2A] \to [BG,B X_{\ms E}]$ is bijective, this determines 
  $f_{\ms E}$ up to homotopy.

To summarize:
\begin{enumerate}
\item $q:EE \to X_{\ms E}$ is a universal $A$-bundle.
\item $BEE \to BX_{\ms E}$ is a universal $BA$-bundle.
\item The inclusions $\iota_G$ and $\iota_{BA}$ are classifying maps.
\item The classifying map $f_{\ms E}: BG \to B^2A$ is determined (up to homotopy) by \eqref{pascal}.
\end{enumerate}


 \subsection{Connecting Homomorphisms} \label{Connecting Homomorphisms} 
In this section, we make a compatibility check between connecting homomorphisms arising from the various classifying spaces associated to an extension.  
Let $A$ be a discrete group, $G$ a nice  well-pointed $t$-group, and  $\ms E=(E,i,\rho) \in {\bf E}(G,A)$. Since $\rho$ is a fibration, we have a connecting homomorphism
\beq
\pi_1(G) \overset{\partial_{\ms E}}{\to} \pi_0(A)=A.
\eeq
We also have  connecting maps
\beq
\pi_1(BA)  \overset{\partial_1}{\to} \pi_0(A) \text{ and } \pi_2(BG)  \overset{\partial_2}{\to} \pi_1(G) 
\eeq
from \eqref{bottle}. (Here $\partial_1=\partial_1^A$ and  $\partial_2=\partial_2^G$.) 
Applying the functor
\beq
\ms B: {\bf E}(G,A) \to \Bun_{BG}(BA)
\eeq
gives a $BA$-bundle $BE \to BG$, and in particular another connecting homomorphism
\beq
\pi_2(BG) \overset{\partial_{BE}}{\to} \pi_1(BA)=A
\eeq

 \begin{prop}  The sequence 
 \begin{equation} \label{ellipse}
 \pi_2(BG) \overset{\partial_2 \times \partial_{BE}}{\to} \pi_1(G) \times \pi_1(BA) \overset{\partial_{\mc E}+\partial_1}{\to} A 
 \end{equation}
is exact at $\pi_1(G) \times \pi_1(BA)$.
\end{prop}

\begin{proof}

 From the LES of homotopy groups we have
 \beq 
 \pi_2(BG) \overset{\partial_{X_{\ms E}}}{\to} \pi_1(G \times BA)  \to \pi_1(X_{\ms E})
 \eeq
exact at $\pi_1(G \times BA)$, with   $\partial_{X_{\ms E}}=\partial_2 \times \partial_{BE}$.
 From the $A$-bundle $(EE,q)$ above we get an isomorphism $\partial_q: \pi_1(X_{\ms E}) \cong A$  of homotopy groups. Since the inclusions of $G$ and $BA$ into $X_{\ms E}$ are classifying maps (Proposition \ref{they're.class}), we obtain commuting squares
 \beq
 \xymatrix{
		\pi_1(G) \ar[r]^{\partial_{\mc E}} \ar[d]  &  A \ar[d]_{=}  \\
		\pi_1(X_{\ms E}) \ar[r]^{\partial_q} & A   \\
	}
\eeq
 and
  \beq
 \xymatrix{
		\pi_1(BA) \ar[r]^{\partial_1} \ar[d]  &  A \ar[d]_{=}  \\
		\pi_1(X_{\ms E}) \ar[r]^{\partial_q} & A   \\
	}
\eeq
 Hence the composition
 \beq
 \pi_1(G) \times \pi_1(BA) \overset{\sim}{\to} \pi_1(G \times BA) \overset{\iota_{\times}}\to \pi_1(X_{\ms E}) \overset{\partial_q}{\to} A
 \eeq
 is equal to $ \partial_{\ms E}+\partial_1$. The proposition follows.
\end{proof}

\begin{cor}\label{anticommute}
	
	The square  
	\beq
	\xymatrix{
		\pi_2(BG) \ar[r]^{\partial_2} \ar[d]_{\partial_{BE}} &  \pi_1(G) \ar[d]^{\partial_{\ms E}}  \\
		\pi_1(BA) \ar[r]^{\partial_1}  & A   \\
	}
	\eeq
	anticommutes, meaning that $\partial_1 \circ \partial_{BE}=- \partial_E \circ \partial_2$.
\end{cor}

\section{Disconnected Groups} \label{disc.section}

In this section we give a cohomology version of the exact sequence \eqref{left.ex.here}.  A more sophisticated approach would be using \cite[Theorem 5.9, page 147]{McCleary}.


\begin{prop} Let $G$ be a nice well-pointed $t$-group. The sequence
\beq \label{H_2exact} 
	 \HH_2(BG^\circ) \xrightarrow{B\iota_*} \HH_2(BG) \xrightarrow{B\qoppa_*} \HH_2(B\pi_0) \to 0
	 \eeq 
	obtained from \eqref{qoppa} is exact.
\end{prop}

\begin{proof}
Since $\qoppa \circ \iota$ is trivial, the induced map from $BG^\circ$ to $B \pi_0$ is constant, hence induces the zero map on cohomology. Next we show that the image of $B\iota_*$ is identical to the image of the Hurewicz map $h: \pi_2(BG) \to \HH_2(BG)$.
Since $BG^\circ$ is simply connected,  we may apply the Hurewicz theorem \cite[Corollary 10.8, page 478]{Bredon} to deduce that $h^\circ: \pi_2(BG^\circ) \to \HH_2(BG^\circ)$ is an isomorphism. By the commutative diagram
\beq 
\xymatrix{ \pi_2(BG^\circ) \ar[r]^{B\iota_*}_{\sim} \ar[d]^{h^\circ}_{\sim} & \pi_2(BG) \ar[d]^{h}\\
\HH_2(BG^\circ) \ar[r]^{B\iota_*} & \HH_2(BG)\\
},
\eeq 
 we deduce that $B\iota_*(\HH_2(BG^\circ)) = h(\pi_2(BG))$. By Lemma \ref{isom.pi.bs} the induced map 
$\pi_1(BG) \to \pi_1(B \pi_0)$ is an isomorphism. Moreover, $\pi_k(B \pi_0)=0$ for $k \geq 2$. The proposition is then a consequence of the following lemma. \end{proof}

\begin{lemma} Let $f:X \to Y$ be a continuous map between path-connected  spaces. Suppose that $f$ induces an isomorphism of fundamental groups, and that
$\pi_2(Y)=\pi_3(Y)=0$.
Then $f_*: \HH_2(X) \to \HH_2(Y)$ is surjective, and its kernel is  the image of the Hurewicz map $h: \pi_2(X) \to \HH_2(X)$.
\end{lemma}

\begin{proof} Let $M_f$ be the mapping cylinder associated to $f$. Then $f$ factors as
\beq
X \hookrightarrow M_f \to Y,
\eeq
and $M_f \to Y$ is a homotopy equivalence.  Note that the relative homotopy sets $\pi_1(M_f,X)$ and $\pi_2(M_f,X)$ are trivial.
By hypothesis, the induced map $\HH_1(X) \to \HH_1(Y)$ is an isomorphism, hence so too is $\HH_1(X) \to \HH_1(M_f)$. It follows from the LES of the pair $(M_f,X)$ that $\HH_1(M_f,X)=0$. 

Therefore by the Relative Hurewicz Theorem \cite[Theorem 10.7, page 478]{Bredon} we deduce that $\HH_2(M_f,X)=0$, and that the Hurewicz map is a surjection
\beq
  \pi_3(M_f,X) \twoheadrightarrow \HH_3(M_f,X).
\eeq
Furthermore, since $\pi_2(M_f)=\pi_3(M_f)=0$, the connecting map $\pi_3(M_f,X) \to \pi_2(X)$ is an isomorphism. The conclusion then follows from the diagram:
\beq
 \xymatrix{ 0 \ar[r] &
 \pi_3(M_f,X) \ar[d] \ar[r]^{\sim}& \pi_2(X) \ar[r] \ar[d] & 0 \\
\HH_3(M_f) \ar[r] & \HH_3(M_f,X)	\ar[r] & \HH_2(X) \ar[r] & \HH_2(M_f) \ar[r] & 0 \\
	}
\eeq 
 \end{proof}

\begin{thm}\label{cohomexact} For an abelian group $A$, there is an exact sequence
\beq
0 \to \HH^2(B \pi_0,A) \overset{B \qoppa^*}{\to} \HH^2(BG,A) \overset{B \iota^*}{\to} \HH^2(BG^\circ,A)^{\pi_0}.
\eeq
\end{thm}

\begin{proof}
By the naturality of the universal coefficient theorem (UCT) we have the following commutative diagram, with exact columns:

\beq \xymatrix{
& 0 \ar[d] & 0 \ar[d] &  \\
 & \Ext^1_{\Z}(\HH_1(B\pi_0),A) \ar[r]^{\sim} \ar[d]  & \Ext^1_{\Z}(\HH_1(BG),A) \ar[r]\ar[d]    & 0 \ar[d]   \\
& \HH^2(B \pi_0,A) \ar[r]^{B\qoppa^*} \ar[d]&  \HH^2(BG,A) \ar[r]^{B\iota^*}  \ar[d] & \HH^2(BG^\circ,A)  \ar[d]   \\
0  \ar[r] &   \Hom_{\Z}(\HH_2(B \pi_0),A)\ar[d]  \ar[r] & \Hom_{\Z}(\HH_2(BG),A) \ar[r]  \ar[d] & \Hom_{\Z}(\HH_2(BG^\circ),A) \ar[d] \\
&   0 & 0 & 0 \\
}
\eeq

The map between the two $\Ext$-groups is an isomorphism by Lemma \ref{isom.pi.bs}.
The $\Hom$-row is left exact by dualizing \eqref{H_2exact}. Therefore $B\qoppa^*$ is injective by the Five Lemma.
Note that $$B\iota^*(\HH^2(BG,A)) \subseteq \HH^2(BG^\circ, A)^{\pi_0}$$ by \cite[Lemma 3.1 Page 55]{Milgram}. 
It is clear that $B\iota^* \circ B\qoppa^* = 0$, and exactness at $\HH^2(BG,A)$ follows from a diagram chase. 
\end{proof}

 \section{From $BA$-Bundles to $\HH^2$} \label{kappa.section}
     Let $B$ be a topological space. A $BA$-bundle $(X,p)$ over $B$  corresponds to a map $f_p: B \to BBA$, unique up to homotopy, with $(X,p)$ the pullback of $EBA \to BBA$ under $f_p$. Recall from Section \ref{ka2} that $B(BA)$ is a $K(A,2)$, and  a ``fundamental class'' ${\bf c} \in \HH^2(BBA,A)$ was defined. 
Then we may set $\kappa(p)=f_{p}^*({\bf c})$. This construction defines a map
\beq
\kappa=\kappa_{A,B}: \bun_B(BA) \to \HH^2(B,A).
\eeq

 \begin{prop}\label{alphabunnatural} The map $\kappa_{A,B}$ is additive, bijective, and natural in both $A$ and $B$.
 \end{prop}
  \begin{proof}
 
  Since $\kappa_{A,B}$ is the composition of the bijection $\bun_{B}(BA)  \overset{\sim}{\to} [B,BBA]$ of \eqref{univ.class.propz} with the bijection 
$  [B,BBA]  \overset{\sim}{\to} \HH^2(B,A)$ of  \cite[Theorem 4.57]{hatcher}, it is itself a bijection.
 
For additivity, let $(X_1,p_1), (X_2,p_2) \in \Bun_B(BA)$, and write $(X_3, p_3)$ for their Baer sum. Let $f_k: B \to BBA$ be classifying maps for $(X_k,p_k)$, with $k=1,2$. Write $\tilde f_k: X_k \to EBA$ for the overmaps. Since $EBA$ is an abelian group, we may define a map
$\tilde f_3: X_1 \times_{B} X_2 \to EBA$ by 
 \beq
 \tilde f_3(x_1,x_2)=\tilde f_1(x_1)+\tilde f_2(x_2).
 \eeq
 Since $\tilde f_3$ is constant on (antidiagonal) $BA$-orbits, it descends to a morphism
 \beq
 \tilde f_3: X_3 \to EBA.
 \eeq
 
 Being $BA$-equivariant, $\tilde f_3$ descends to a map $f_3:B \to BBA$. 
 Now $f_3=f_1+f_2$, and  $f_3$ is a classifying map for $X_3$.
 Thus
 \beq
 \begin{split}
 \kappa(p_3) &= f_3^*({\bf c}) \\ 
 &=f_1^*({\bf c})+f_2^*({\bf c}) \text{ (by Lemma \ref{Hspace})} \\
 &= \kappa(p_1) + \kappa(p_2) ,
  \end{split}
  \eeq
  which shows additivity. It is easy to see that $\kappa$ is natural in $B$. 
 For naturality in $A$, one uses Lemma \ref{roti.sabzi}; we omit the details.
   
 \end{proof}
 
  \section{From Extensions to $\HH^2$} \label{alpha.section}
  Finally we reach our main objective: the definition of the correspondence $\alpha$, its agreement with $-\beta$, and Theorem \ref{mid.intro}.
  \subsection{Definition of $\alpha_G$}
  Let   $G$ be a nice well-pointed $t$-group. Write $\alpha=\alpha_{G,A}$ for the 
  composition
  \beq
  \mb E(G,A) \overset{\ms B}{\to} \bun_{BG}(BA) \overset{\kappa}{\to} \HH^2(BG,A).
  \eeq
  
  By Propositions  \ref{B.map.inj} and \ref{alphabunnatural},  and the results of Section \ref{rain.car}, we deduce:
  \begin{prop}  \label{G.disc.alpha.inj} 
  The map $\alpha_{G,A}$ is a homomorphism of abelian groups, natural in $G$ and $A$. 
  \end{prop}

 
 \subsection{Case of Cocycle Extensions} \label{cocycle.subs}

\begin{defn} A $2$-cocycle is a (continuous) function $z: G \times G \to A$ satisfying $z(g,1)=0=z(1,g)$ and
\beq
z(g,h)+ z(gh,k)=z(h,k)+z(g,hk)
\eeq
for all $g,h,k \in G$. Write $Z^2(G,A)$ for the additive group of $2$-cocycles.
\end{defn}

Given $z \in Z^2(G,A)$, let us recall the corresponding central extension $\ms E_z=(E,p,i)$ of $G$ by $A$.
One takes $E=G \times A$ as a space, and defines multiplication by
\begin{equation} \label{cocyk}
(g,a) \cdot (h,b)=(gh, a+b+z(g,h)).
\end{equation}
Let $p$ be the first projection, and let $i(a)=(1,a)$. Then $(E,p)$ is a trivial $A$-bundle.

Recall from Section \ref{Milgram.Steenrod.sec} that a typical element of $B_2G$ is written as $[g_1,t_1][g_2,t_2]$ for $g_1,g_2 \in G$ and $0 \leq t_1 \leq t_2 \leq 1$, and a typical element of $B^2A$ is written as
$[a;t_1,t_2]$, with $a \in A$ and $t_1,t_2 \in I$. We also had a homeomorphism $\eta_2: \Delta_2/\partial \Delta_2 \cong I^2/\partial I^2$ given by $\eta_2(t_1,t_2)=(\frac{t_1}{t_2},t_2)$.

 \begin{prop} \label{defn.phi.here}
 There is a classifying map $f_z: BG \to B^2A$ for $\ms E_z$ whose restriction to $B_2G$ is given by
 \beq
 \phi_z :[g_1,t_1][g_2,t_2] \mapsto   \left[z(g_1,g_2);\eta_2(t_1,t_2) \right], 
 \eeq
 and taking $B_1G$ to $1_{B^2A}$.
 \end{prop}

 \begin{proof}

 Recall from \eqref{suspend.disbelief} the surjection  $q: B_2G \to \Sigma^2(G \wedge G)$.  Consider the diagram
 
   \beq 
	\xymatrix{
		Z^2(G,A) \ar[d]^{} \ar[r]^{}   &[BG,B^2A] \ar[d]^{} \\
		\HH^0(G \wedge G,A) \ar[d]^{\Sigma^2}     &[B_2 G,B^2A]\\
		 \HH^2(\Sigma^2(G \wedge G),A) \ar[ur]^{q^*} .}
	\eeq

  The top horizontal map takes a cocycle $z$ to a classifying map $f_z$ for the $BA$-bundle $BE \to BG$. The right vertical map is induced by the inclusion $B_2G \to BG$.  Traversing down twice, then right   takes $z$ to $\phi_z$. Hence it is enough to show that the diagram commutes. For this, we extend the diagram to the right by postcomposing with the weak homotopy equivalence $B \iota_{BA}$, with $\iota_{BA}$ as in Section \ref{x.e.sec}:
    \beq 
	\xymatrix{
		 [BG,B^2A] \ar[d] \ar[r] & [BG,BX_{\ms E_z}] \ar[d] \\
		 [B_2 G,B^2A]  \ar[r] & [B_2G,BX_{\ms E_z}]\\}
	\eeq 
Write $R_z$ for the restriction of $(B \iota_{BA})_* \phi_z$ to $B_2G$; it is given by
\beq
R_z:  [g_1,t_1][g_2,t_2] \mapsto [1_{EG} \times [z(g_1,g_2), t_1/t_2],t_2]. 
  \eeq
Here $[z(g_1,g_2),t_1/t_2] \in B_1A \subset BE$, so that $1_{EG} \times [z(g_1,g_2),t_1/t_2] \in X_{\ms E}$, and so  $[1_{EG} \times [z(g_1,g_2), t_1/t_2],t_2] \in B_1X_{\ms E}$. On the other hand, by \eqref{pascal},  $(B \iota_{BA})_* f_z$ is homotopic to $B \iota_G$. Write $L_z$ for the restriction of $B\iota_G$ to $B_2G$; it is given by
  \beq
L_z: [g_1,t_1][g_2,t_2] \mapsto [g_1 \times 1_{BE},t_1][g_2 \times 1_{BE},t_2] \in B_2 X_{\ms E_z}. 
  \eeq

The diagram commutes if  $L_z$ and $R_z$ are homotopic.  Write $\xi: G \times I \to X_{\ms E_z}$ for the map 
\beq
\xi(g,s)=(g,s) \times [g,s].
\eeq

Given $s \in I$, consider the following formula for a homotopy $H_s:B_2G \to BX_{\ms E_z}$:
\beq
[g_1,t_1][g_2,t_2] \mapsto \left[\xi(g_1,s),t_1 \right] \left[(g_2,s) \times [g_2,s] \left[z^{-1},s \frac{t_1}{t_2} \right] \left[z,\frac{t_1}{t_2} \right],t_2 \right],
\eeq
where $z=z(g_1,g_2)$. Some comments are in order to see this is well-defined. First, note that
\beq
\begin{split}
H_s([g_1,t][g_2,t]) &=[(g_1g_2,s) \times [g_1,s][g_2,s][z^{-1},s],t] \\
				&= [\xi(g_1g_2,s),t] \\
				&= H_s([g_1g_2,t][1,1]), \\
				\end{split}
				\eeq
where we have used \eqref{cocyk} for the second equality. Since $z(1,g)=z(g,1)=0$, we have
\beq
H_s([g,t][1,t'])=[\xi(g,s),t] =H_s([1,t'][g,t]).
\eeq
It is easy to see  also that for all $g,g' \in G$ and $s,t \in I$ we have $H_s([g',0][g,t])=[\xi(g,s),t]=H_s([g,t][g',1])$. Hence $H_s$ is well-defined.  Finally note that $H_0=R_z$ and $H_1=L_z$, so this gives the desired homotopy.

 \end{proof}

\subsection{Case of $G$ discrete}

 \subsubsection{Bar Construction and Group Cohomology} \label{bar.const.subsub}
 Let $G$ be discrete. Let us recall a version of the ``bar'' construction for a classifying space of $G$. (See for example \cite[Example 1B.7]{hatcher}.)
A $\Delta$-complex $\bar EG$ is defined with $n$-simplices corresponding to $G^{n+1}$. 
We write $\lip g_0, \ldots, g_n \rip$ for a typical $n$-simplex; there is an obvious way to attach  it to the $(n-1)$-simplex $\lip g_0, \ldots, \hat g_i, \ldots, g_n \rip$. Write $\Delta^n$ for the simplex defined as the points $P=(s_0, \ldots, s_n) \in I^{n+1}$ with $\sum_{i=0}^n s_i=1$.
A general point in $\bar EG$, lying in the cell corresponding to $\lip g_0,\ldots, g_n \rip$, is specified by a point  $P\in \Delta^n$ as above. We will represent this point with the   notation $(g_0, g_1, \ldots, g_n; P)$.


For a right action of $G$ on $\bar EG$ we take the diagonal right multiplication, applied to each $G$-coordinate. Write $\bar BG$ for the quotient.  Then $\bar EG$ is contractible, and so $\bar BG$ is a classifying space for $G$. Write $\bar E_nG$ for the $n$-skeleton on $\bar EG$, being the union of all $k$-simplices with $k \leq n$, and $\bar B_nG$ for the image of $\bar E_nG$ under the quotient. A typical member of $\bar B_nG$ is expressed as $[g_0,g_1, \ldots, g_n;P]$ with $P \in \Delta^n$. The simplices of $\bar BG$ are given by $[g_0, \ldots,g_n]:=\{[g_0, \ldots, g_n;P]:P \in \Delta^n\}$.
 
For understanding $\bar BG$ as a $\Delta$-complex, the $\Delta^n$ form of the simplex is preferable. For understanding it as a CW-complex, the $\Delta_n$-form is preferable. So we fix isomorphisms between them. Let $\theta^n: \Delta^n \to \Delta_n$ be the isomorphism $(s_0,\ldots, s_n) \mapsto (s_0,s_0+s_1, \ldots, s_0+s_1+ \cdots + s_{n-1})$; its inverse
$\theta_n: \Delta_n \to \Delta^n$ is given by $(t_1, \ldots, t_n) \mapsto (t_1, t_2-t_1, \ldots, t_n-t_{n-1},1-t_n)$.

Given a tuple $e=(g_0, g_1, \ldots, g_n)$, we have a characteristic map $\Phi_e: \Delta_n \to \bar BG$ given by $\ul t \mapsto [e;\theta_n(\ul t)]$.
 This gives a CW-structure to $\bar BG$.
 
We can form a free resolution $F_\bullet$ of the trivial $G$-module $\Z$ as follows. Let $G$ act on $G^{n+1}$ by right multiplication on the last coordinate.
Let $F_n$ be the resulting permutation representation (over $\Z$) of $G$. Thus an abelian group, $F_n$ is the free abelian group on $G^{n+1}$, and its $G$-module structure is defined by the above action on the $\Z$-basis $\{(g_0,g_1, \ldots, g_n) \mid g_i \in G\}$. Then $F_n$ is a free $\Z G$-module with $\Z G$-basis $(g_1, g_2, \ldots, g_n,1)$. We have differentials $d_n: F_n \to F_{n-1}$ given by $d_1((g,1))=1-g$, and 
\beq
d_n((g_0,\ldots, g_n)) = (g_1, \ldots, g_{n}) +        \sum_{i=1}^{n-1} (-1)^i (g_0,g_1, \ldots, g_{i-1} g_{i}, \ldots, g_n) + (-1)^n (g_0, \ldots, g_{n-1})
\eeq
for $n \geq 2$.

 
  A $G$-equivariant bijection of the $\Z$-basis of $F_n$ with the $n$-simplices of $\bar EG$ is given by sending $(g_0, \ldots, g_n)$ to   $\lip g_0 \cdots g_n, g_1 \cdots g_n, \ldots, g_n \rip$. This induces a $\Z G$-module isomorphism of $F_n$ with the group $C_n^\Delta(\bar EG)$ of simplicial $n$-chains of $\bar EG$.

A  $G$-module morphism $F_n \to A$ is constant on $G$-orbits, and the group of $G$-equivariant simplicial cocycles in $C^n_\Delta(\bar EG,A)$ is isomorphic to $C_\Delta^n(\bar BG,A)$. Hence from the above we obtain an isomorphism
\begin{equation} \label{pran.5.8}
\sigma_n: \Hom_G(F_n,A) \overset{\sim}{\to} C_\Delta^n(\bar BG,A).
\end{equation}
Below we have two complexes, with the upper producing $\HH_{\gp}^*(G,A)$ and the lower producing $\HH^*(\bar BG,A)$.  

 \beq \xymatrix{
  \cdots \ar[r] & \Hom_G(F_n,A)  \ar[d]^{\approx} \ar[r]  & \Hom_G(F_{n+1},A) \ar[r] \ar[d]^{\approx} & \cdots  \\
  \cdots \ar[r] &C_\Delta^n(\bar BG,A) \ar[r] & C_\Delta^{n+1}(\bar BG,A) \ar[r] & \cdots  \\
}
\eeq
 We write $\sigma_n:\HH_{\gp}^n(G,A) \overset{\sim}{\to} \HH^n(\bar{B}G,A)$ for the resulting isomorphism.

 A basis of $F_n$ as a $\Z G$-module is $G^n \times \{ 1_G \} \subset G^{n+1}$, so $\Hom_G(F_n,A)$ is isomorphic to the abelian group of set-maps $G^n \to A$. We will use this identification throughout; thus  for $f: G^n \to A$ we have $\sigma_n(f) \in C_\Delta^n(\bar BG,A)$ defined by
 \beq
 \sigma_n(f)([ g_0, \ldots, g_n])=f(g_0g_1^{-1}, \ldots, g_{n-1} g_n^{-1}).
 \eeq

 \subsubsection{Bar Construction and the Milgram-Steenrod Construction}
 

 The maps $\tilde \Psi_j: G^{j+1} \times \Delta^j \to G^{j+1} \times \Delta_j$, for $0 \leq j \leq n$, defined by
 \beq
 \tilde \Psi_j(g_0, \ldots, g_j;P)= (g_0g_1^{-1},  g_1 g_2^{-1}, \ldots,g_j; \theta^j(P))
 \eeq
  induce  $\Psi_n: \bar{E}_nG \to E_nG$. Write $\Psi:\bar E G \to EG$ for the colimit of the $\Psi_n$; it
 descends to a map $\bar \Psi: \bar BG \to BG$.  
 
 Note that $\bar \Psi_2 (g_0,g_1,g_2; s_0,s_1)=[g_0g_1^{-1},s_0][g_1g_2^{-1},s_0+s_1] \in B_2G$.

 \begin{prop} The map $\bar \Psi$ is a homotopy equivalence.
 \end{prop}
 
 \begin{proof}
 Since $BG$ and $\bar BG$ are both classifying spaces for $G$, it suffices to show that $\Psi$ induces an isomorphism of fundamental groups.
 Let $g \in G$. The loop $\ol \gm_g: t \mapsto [g,t,1_G,1-t]$ in $\bar BG$ is sent to the loop $\gm_g: t \mapsto (g,t)$ in $BG$. The unique lift
$ \widetilde{\ol \gm_g}$ of $\ol \gm_g$ to $\bar EG$ starting at $[1_G,1]$ has endpoint $[g,1]$, and likewise the unique lift $\tilde \gm_g$ of $\gm_g$ to $EG$ starting at
$1_{EG}$  has endpoint $(g,1)$. Hence the following diagram commutes:

\beq
 \xymatrix{
\pi_1(\bar BG) \ar[r]^{\partial} \ar[d]^{\bar \Psi} & \pi_0(G)=G  \ar@{=}[d] \\
\pi_1(BG) \ar[r]^{\partial} & \pi_0(G)=G, \\
}
\eeq
 so this is indeed an isomorphism.

 \end{proof}
 
Recall from Proposition \ref{defn.phi.here} that to a $2$-cocycle $z \in Z^2(G,A)$ we associated a map $\phi_z:B_2G \to B^2A$.
Then $\bar \phi_z:=\bar \Psi_2^*(\phi_z): \bar B_2G \to B^2A$ is given by
 \begin{equation} \label{capital1}
  [g_0,g_1,g_2;P] \mapsto \left[z(g_0g_1^{-1},g_1 g_2^{-1}); \eta_2(\theta^2(P)) \right].
 \end{equation}
 
 Let $\Cl: \HH^2_{\gp}(G,A) \to  \mb E(G,A)$ be the morphism sending a $2$-cocycle $z$ to the central extension $\ms E_z$ as in Section \ref{cocycle.subs}. Since $G$ is discrete, $\Cl$ is classically known to be an isomorphism. Let $\cl$ be the inverse of $\Cl$.
 
   \begin{thm} \label{alpha.agrees} For $G$ discrete, the diagram
\beq  \xymatrix@R=1pc{
     &  \HH^2_{\gp}(G,A)  \ar[dd]^{(\Psi^*)^{-1} \circ \sigma_2} \\
    \mb E(G,A) \ar[ur]^{\cl} \ar[dr]^\alpha & \\
   & \HH^2(BG,A)
}
  \eeq
  of isomorphisms commutes.  (Thus $\alpha$ can be identified with $\cl$.)
  \end{thm}
  
(Compare \cite[Lemma 1.12]{Milgram}.)

\begin{proof}
 Consider the diagram
\beq \xymatrix{
  z \in  \HH^2_{\gp}(G,A) \ar[r]_\sim^{\Cl}  \ar[d]^{\sigma_2}  & \mb{E} (G,A) \ni \ms E_z \ar[d]^{\alpha}      \\
  \HH^2_\Delta(\bar BG,A) \ar[d] & \HH^2(BG,A) \ar[d] \\
\bar f_z \in [\bar BG,B^2A]     \ar[d]^{\bar r_2}&  [ BG,B^2A]  \ni f_z \ar[l]_{\Psi^*}^\sim  \ar[d]^{r_2}    \\
\bar \phi_z \in [\bar B_2G,B^2A]     &   [B_2G,B^2A] \ni \phi_z \ar[l]_{\Psi^*_2}      \\
}
\eeq
Here $\ol r_2$ and $r_2$ are restriction maps from $\bar BG$ and $BG$ to their $2$-skeleta  $\bar B_2G$ and $B_2G$. As such, they are injective.

Let us see that the outer square commutes. Let $z \in Z^2(G,A)$. According to Lemma \ref{jun.19.lem}, 
\beq
\begin{split}
 \bar f_z ([g_0,g_1,g_2;P]) &= \phi_f (\Phi_{g_0,g_1,g_2}(\theta^n P)) \\
&= [z(g_0g_1^{-1},g_1g_2^{-1});\eta_2(\theta^2(P))].\\
\end{split}
\eeq
So $\bar r_2(\bar f_z)=\phi_z$, meaning the outer square commutes.

 Clearly the bottom square commutes, and since $r_2$ is injective, it must be that the top square commutes. Since all morphisms in the top square, not including $\alpha$, are isomorphisms, it must be that $\alpha$ is an isomorphism.
 \end{proof}
 
\begin{cor} \label{G.disc.alf} When $G$ is discrete, the map $\alpha=\alpha_{G,A}$ is an isomorphism.
\end{cor}

As in \cite{hatcher}, the $n$-cells $e$ of a CW-complex correspond to maps $\Phi_e: \Delta_n \to X$. Write $X^{(n)}$ for the $n$-skeleton of $X$. 
 Recall the homeomorphisms $\eta_n: \Delta_n/\partial \Delta_n \overset{\sim}{\to}  I^n/\partial I^n$ from Section \ref{EG.gp}.
 
 \begin{lemma} \label{jun.19.lem} Let $X$ be a CW-complex, and $f \in Z^n_{CW}(X,A)$. We may regard $f$ as a homomorphism from $\HH_n^{CW}(X^{(n)}/X^{(n-1)})$ to $A$. Consider the map $\phi_f: X^{(n)} \to B^nA$ defined by sending $X^{(n-1)}$ to $1_{B^nA}$ and, for each $n$-cell $e$ and $\ul t \in \Delta_n$,
 \beq
 \phi_f(\Phi_e(\ul t)):=[f(e);\eta_n(\ul t)].
 \eeq

 Then $[\phi_f] \in [X^{(n)},B^nA]$ is the image of $f$ under the composition
 \beq
 Z^n_{CW}(X,A) \to \HH^n_{CW}(X,A) \to                   [X,B^nA] \to [X^{(n)},B^nA].
 \eeq
 \end{lemma}
 \begin{proof}

 For an $n$-cell $e$ of $X$, the composition $\phi_f \circ \Phi_e: \Delta_n \to B^nA$ maps the boundary to $1$, hence we may compose 
 \beq
 \phi_f \circ \Phi_e \circ \eta_n^{-1}: I_n/\partial I_n \to \Delta_n/\partial \Delta_n  \to B^nA.
 \eeq
 This map is exactly  $\gm_{f(e)}$, where $\gm_a$ is defined in Lemma \ref{gamma.a}. The isomorphism $[X^{(n)},B^nA]$ maps $[\phi_f]$ to $\phi_f^*(\iota_{B^nA}) \in \HH^n(X^{(n)},A)$.
 
  Recall the   surjection $\beta: \HH^k(X,A) \to \Hom(\HH_k(X),A)$ from the Universal Coefficient Theorem.
  When $X$ is a CW-complex of dimension $n$, then 
  $c \in \HH^n(X,A)$ is determined by $\lip \beta(c),\Phi_e \rip$, where $e$ runs over the $n$-cells of $X$.
 So it suffices to check that $\lip \beta( \phi_f^*(\iota_{B^nA})), \Phi_e \rip=f(e)$. But this is clear.
  \end{proof}

 \subsection{Case of $G$ connected}
 Recall from Definition  \ref{beta.defn} the isomorphism $\beta_G:  \mb E(G,A) \overset{\sim}{\to} \HH^2(BG,A)$, valid for nice connected well-pointed $t$-groups $G$.

\begin{prop} \label{alpha.beta}
	When $G$ is connected, $\alpha_G=-\beta_G$.
\end{prop}
\begin{proof}
	
	Let $\ms E=(E,i,p) \in {\bf E}(G,A)$.
	With the above construction we have 
	\beq
	\alpha_G(\ms E)= f_{Bp}^*(\iota) \in \HH^2(BG,A),
	\eeq
	where $\iota \in \HH^2(BBA,A)$ was introduced in Section \ref{ka2}.
Moreover the extension gives   $\pi_2(f_{Bp}):\pi_2(BG) \to   \pi_2(BBA)= A$. 
From the long exact sequences of homotopy groups we obtain:
\begin{equation} \label{ashish} 
      \xymatrix{
  \pi_2(BG) \ar[r]^{ \partial_{BE}} \ar[d]^{\pi_2(f_{Bp})}& \pi_1(BA)  \ar[d]^{||} \\
 \pi_2(BBA) \ar[r]^{\partial_{BEA}} & \pi_1(BA) \\
}
\end{equation}
 
Recall from \eqref{characterized} that $\beta_G$ is characterized by $m^2(\beta_G(\ms E))=\partial_{\ms E}$. Therefore we must prove that
\beq
h^2 q^2(f_{Bp}^*(\iota_A))=-\partial_{2}^*( \partial_E).
\eeq

We have	
\beq
\begin{split}
h^2 q^2(f_{Bp}^*(\iota_A)) &= \pi_2(f_{Bp})^*(\zeta) \text{ by \eqref{m2.f.iota}} \\
					&=\zeta \circ \pi_2(f_{Bp}) \\
			&= \partial_1 \circ \partial_{BEA} \circ \pi_2(f_{Bp})  \\
			&= \partial_1 \circ \partial_{BE} \text{ by \eqref{ashish} } \\
			&=  - \partial_E \circ \partial_2 \text{ (by Corollary \ref{anticommute})} \\
			&= -\partial_{2}^*( \partial_E), \\
\end{split}
\eeq
as required. \end{proof}
 
   \subsection{A Criterion to Ensure $\alpha_G$ is an isomorphism}

We continue to assume that $G$ is a nice well-pointed $t$-group.
   
\begin{thm} \label{big.theorem}
The map $\alpha_G$ is injective. In both of the following cases, $\alpha_G$ is an isomorphism:
\begin{enumerate}
\item $G$ is the semidirect product of a discrete group and a connected group.
\item $G^\circ$ is simply connected.
\end{enumerate}

\end{thm}


\begin{proof}
 
  Put $\pi_0=\pi_0(G)$, and consider the diagram
\beq
\xymatrix{
0 \ar[r] & \mb E(\pi_0,A)\ar[d]^{\alpha_{\pi_0}} \ar[r]^{\qoppa^*} & \mb E(G,A) \ar[d]^{\alpha_G} \ar[r]^{\iota^*} & \mb E(G^\circ, A)^{\pi_0} \ar[d]^{\alpha_{G^\circ}}  \\
0 \ar[r] & \HH^2(B\pi_0,A)\ar[r]^{B\qoppa^*} & \HH^2(BG,A) \ar[r]^{B\iota^*} & \HH^2(BG^\circ, A)^{\pi_0}     }
\eeq	
 
 The top row is exact by Proposition \ref{extexact},  and the bottom row is exact by Theorem \ref{cohomexact}. The diagram is commutative by the naturality of $\alpha$.
 The map $\alpha_{\pi_0}$ is an isomorphism by Corollary \ref{G.disc.alf}. The map $\alpha_{G^\circ}$ is a bijection by Theorem \ref{Gconnect.equiv} and Proposition \ref{alpha.beta}. The injectivity of $\alpha_G$ then follows by the 5-Lemma.  

When $G = G^\circ \rtimes \pi_0$, then $\iota^*$ is surjective by Proposition  \ref{surj}. From another diagram chase we deduce that $\alpha_G$ is an isomorphism. 
\end{proof}

 
 \begin{remark} It is also true that $\alpha$ is an isomorphism when   $\HH^3_{\gp}(\pi_0(G),A)=0$; see \cite[Theorem 8.0.6]{Pranjal.Thesis}.
 \end{remark}

  \subsection{Lifting and Cohomology}

\begin{thm} \label{lifting.and.coh} Let $\varphi:G' \to G$ be a homomorphism between   nice well-pointed $t$-groups, and $\ms E=(E,i,p) \in {\bf E}(G,A)$. Then $\varphi$ lifts to $E$ iff 
\beq
\varphi^*(\alpha_G(\ms E))=0.
\eeq
\end{thm}

\begin{proof} 
From Proposition \ref{ext.lift}, we know $\varphi$ lifts to $E$ iff $\varphi^*(\ms E)=0$. But as $\alpha_{G'}$ is injective (Theorem \ref{big.theorem}), this is equivalent to the identity 
\beq
0= \alpha_{G'}(\varphi^* \ms E) =\varphi^* (\alpha_G(\ms E)).
\eeq \end{proof}

 
 \section{Stiefel-Whitney Classes} \label{SWC.section} 
 
 In this section we define Stiefel-Whitney classes for orthogonal complex representations. This is equivalent to the definition in \cite[Section 2.6]{Benson.II} given for real representations, when $G$ is compact. For the convenience of the casual reader, let us take $G$ to be a Lie group, although this could be generalized to ``nice well-pointed $t$-group''.

 Let $V$ be a complex vector space of dimension $n$, with a nondegenerate quadratic form $Q$. Write $\Or(V)$ for the corresponding orthogonal group. 
 Let $\mc B=\{e_1, \ldots, e_n\}$ be an orthonormal basis of $V$. Then
 \beq
\Gamma=\{ g \in \Or(V) \mid \forall i, ge_i= \pm  e_i \}
 \eeq
is evidently an $\mb F_2$-vector space with an obvious basis.  Write $v_1, \ldots, v_n$ for the dual basis, viewed in  $\HH^1(B\Gamma,\Z/2\Z)$. Then
 \beq
 \HH^*(B\Gamma,\Z/2\Z) =\Z/2\Z[v_1, \ldots, v_n],
 \eeq
 meaning the mod $2$ cohomology of $B \Gamma$ is a polynomial algebra in   the $v_i$. The restriction map
 \beq
 \HH^*(\BO(V), \Z/2\Z) \to  \HH^*(B\Gamma, \Z/2\Z)
 \eeq
is injective \cite[Theorem 2.2]{toda}, and its image is the symmetric algebra   $\Z/2\Z[v_1, \ldots, v_n]^{S_n}$. Write $\mc E_1, \ldots, \mc E_n$ for the elementary symmetric polynomials in the $v_i$; by the Fundamental Theorem of Symmetric Polynomials, we can therefore identify
\beq
 \HH^*(\BO(V), \Z/2\Z) \cong \Z/2\Z[\mc E_1, \ldots, \mc E_n].
 \eeq
 For $1 \leq k \leq n$, write $w_k \in  \HH^k(\BO(V), \Z/2\Z)$ for the cohomology class corresponding to $\mc E_k$.  
 
 Now let $G$ be a Lie group, and $\pi: G \to \Or(V)$ an orthogonal (complex) representation. We define
 \beq
 w_k(\pi)=\pi^*(w_k) \in \HH^k(BG,\Z/2\Z);
 \eeq
 this is called the $k$th Stiefel-Whitney Class (SWC) of $\pi$. Note that if  $\varphi: G_1 \to G_2$ is a homomorphism of Lie groups, and $\pi$ is an orthogonal representation of $G_2$, then  
  \begin{equation} \label{funct.w}
  \varphi^*(w(\pi))=w(\pi \circ \varphi).
  \end{equation}
 Recall the definition of $m^1$ from Proposition \ref{m^1}.
 
  \begin{thm} Let $\pi: G \to \Or(V)$ be an orthogonal representation of a Lie group $G$. Then the isomorphism $m^1: \HH^1(BG,\Z/2\Z) \to \Hom_c(G, \mu_2)$ takes $w_1(\pi)$ to $\det  \pi$. \end{thm}

\begin{proof}
Since $\HH^1(\BO(V), \Z/2\Z)=\{0,w_1\}$, the isomorphism 
\beq
m^1:  \HH^1(BG,\Z/2\Z) \to \Hom_c(\Or(n), \mu_2)
\eeq
 must take $w_1$ to $\det$.
 Now by the commutative square
   \beq
\xymatrix{ 
		 \HH^1(\BO(V),\Z/2\Z) \ar[r]^{m^1} \ar[d]^{\pi^*}&  \Hom_c(\Or(V),\Z/2\Z) \ar[d] \\
		 \HH^1(BG,\Z/2\Z) \ar[r]^{m^1} & \Hom_c(G,\Z/2\Z) \\
	}
\eeq
we find that
 \beq
 \begin{split}
 m^1(w_1(\pi)) &= m^1(\pi^* w_1) \\
 		&= \det \pi, \\
		\end{split}
		\eeq
		as required.
\end{proof}

   \section{Extensions of $\Or(V)$} \label{double.covs.ort.sec}
 
 Assume now that $\dim V \geq 2$. Since 
 \beq
 \HH^2(\BO(V),\Z/2\Z) =\{0,w_1^2, w_2, w_1^2+w_2 \},
 \eeq
 there are four inequivalent central extensions of $\Or(V)$ by $A=\Z/2\Z$. Let us undertake to describe these extensions.
 
 The first one is easy to describe. Taking determinants gives a surjection from $\Or(V)$ to $\mu_2$. Of course, $\mu_2$ has an $A$-cover, by squaring the cyclic group $\mu_4 < \mb C^\times$ generated by the imaginary unit $i$. 
  Write $\widetilde \Or(V) \to \Or(V)$ for the pullback:
   \beq 
	\xymatrix{
		\widetilde \Or(V)   \ar[d]   \ar[r]  &  \mu_4  \ar[d]^{z \mapsto z^2}\\
		 \Or(V)  \ar[r]^{\det} & \mu_2}
	\eeq 
 Thus $\widetilde \Or(V)$ is the subgroup of  pairs $(g,z) \in \Or(V) \times \mu_4$ with $\det g=z^2$.

  \subsection{Pin Groups}
 
Let $V$ be a finite-dimensional (complex) vector space, with a nondegenerate quadratic form $Q$.  
Say $u \in V$ is a \emph{unit vector} provided $Q(u)=1$, and an \emph{antiunit vector}, provided $Q(u)=-1$. 
Note that if $u$ is a unit vector, then $iu$ is an antiunit vector.

The Clifford algebra $C(V)$ is the quotient of the tensor algebra $T(V)$ by the two-sided ideal generated by the set
\beq
 \{ v\otimes v -Q(v): v\in V\}.
 \eeq
 It contains $V$ as a subspace.  Write $C(V)^\times$ for the group of invertible members of $C(V)$. 
 
 \begin{defn} Write $\Pin^-(V)$ for the subgroup of $C(V)^\times$ generated by the antiunit vectors in $V \subset C(V)$, and
 $\Pin^+(V)$ for the subgroup generated by the unit vectors.
 \end{defn} 
  
 There are unique homomorphisms  $\rho^-: \Pin^-(V) \to \Or(V)$ and $\rho^+:  \Pin^+(V) \to \Or(V)$. Here $\rho^-$ (resp., $\rho^+$) sends an antiunit (resp., unit) vector $u$ to the reflection of $V$ determined by $u$. 
 The kernels of $\rho^+$ and $\rho^-$ are both equal to $\mu_2$.
 Since $\Or(V)$ is generated by reflections, $\rho^+$ and $\rho^-$ are surjective. Thus they are extensions of $\Or(V)$ by $\mu_2$.    See \cite[Chapter 20]{fulton} and \cite[Appendix 1]{frohlich} for details.

 \subsection{Restriction to $\Gamma_2$}
 Now suppose $\dim V \geq 2$. Pick perpendicular unit vectors $e_1,e_2 \in V$, with corresponding reflections $r_1,r_2$.  Let $\Gamma_2< \Or(V)$ be the Klein $4$-subgroup generated by $r_1,r_2$. 
 We have:
    \beq 
	\xymatrix{
		\mb E(\Or(V),\Z/2\Z)   \ar[d]^{\alpha_G}   \ar[r]  &  \mb E(\Gamma_2,\Z/2\Z) \ar[d]^{\alpha_{\Gamma_2}}\\
		\HH^2(\BO(V),\Z/2\Z)  \ar[r] & \HH^2(B\Gamma_2,\Z/2\Z)       }
	\eeq
 
Since $\Gamma_2$ is discrete, the bijection $\alpha_{\Gamma_2}$ is well-understood; see for instance  \cite[page 830]{Dummit}. 
For each of the four extensions $\ms E$ above, one determines the restriction $\ms E_2$ of $\ms E$ to $\Gamma_2$, and computes
$\alpha_{\Gamma_2}(\ms E_2)$ explicitly with $2$-cocyles. Here is the result:

\begin{prop} Let $G$ be a Lie group. An orthogonal representation $\pi: G \to \Or(V)$ lifts to
\begin{enumerate}
\item $\widetilde \Or(V)$ iff $w_1(\pi)^2=0$.
\item $\Pin^+(V)$ iff $w_2(\pi)=0$.
\item $\Pin^-(V)$ iff $w_2(\pi)+w_1^2(\pi)=0$.
\end{enumerate}
The set of lifts of $\pi$ is acted on simply transitively by the group of continuous linear characters $\chi: G \to \mu_2$. 
\end{prop}

\begin{proof}
 This follows from the above and Theorem \ref{lifting.and.coh}.
 \end{proof}
 
 \section{Further Directions} \label{sec:further}
 
There is scope for extending the theory of this paper to any degree, and to the case where the abelian group $A$ is not discrete.  In \cite{Pranjal.Thesis}, the first author studied the case when $A$ is not discrete; there is a generalized cohomology theory $\HH^*(-,A)$ and a natural morphism
$\alpha: \mb E(G,A) \to \HH^2(BG,A)$, generalizing the one in this paper.
Moreover, he also defined natural maps $\alpha^n:  \HH^n_c(G,A) \to \HH^n(BG,A)$, where $\HH^n_c(G,A)$ denotes continuous group cohomology.
When $n=2$,  this agrees with a restriction of $\alpha$ above. When $G,A$ are discrete, this agrees with $\sigma_n$ in Section \ref{bar.const.subsub}. Likely one could define an analogue of $\alpha$ from Yoneda extensions of length $n$ to $\HH^n(BG,A)$; see  \cite[Section 10.3]{Pranjal.Thesis} for a discussion.
 For example, an approach to prove $\alpha$ is an isomorphism (with $A$ discrete) in the spirit of the argument for Theorem \ref{big.theorem} is discussed on \cite[page 112]{Pranjal.Thesis}.

   \bibliographystyle{alpha}
\bibliography{mybib}
 \end{document}